\documentclass[10pt]{article}
\usepackage{color}
\usepackage{amssymb}
\usepackage{amsthm,array,amssymb,amscd,amsfonts,latexsym, url}
\usepackage{amsmath}
\usepackage[all]{xy}

\newtheorem{theo}{Theorem}[section]
\newtheorem{sublemma}[theo]{Sublemma}
\newtheorem{prop}[theo]{Proposition}
\newtheorem{lemm}[theo]{Lemma}
\newtheorem{coro}[theo]{Corollary}
\newtheorem{rema}[theo]{Remark}
\newtheorem{Defi}[theo]{Definition}

\newtheorem{example}[theo]{Example}
\title{Unirational threefolds with no universal \\codimension $2$ cycle}
\author{Claire Voisin
\\CNRS and \'{E}cole Polytechnique}
\date{}

\newfont{\gothic}{eufb10}
\begin{document}
\maketitle
\setcounter{section}{-1}

\begin{abstract}   We prove that the general  quartic double
solid with $k\leq 7$ nodes
   does not
admit a Chow theoretic  decomposition of the diagonal, (or
equivalently  has a nontrivial universal ${\rm CH}_0$ group,) and
the same holds if we replace in this statement ``Chow theoretic'' by
``cohomological''. In particular, it is not stably rational. We also
deduce  that the general quartic double solid with seven nodes does
not admit a universal codimension $2$ cycle parameterized by its
intermediate Jacobian, and even does not admit a parametrization
with rationally connected fibers of its Jacobian by a family of
$1$-cycles. This finally implies that its third unramified
cohomology group is not universally trivial.
\end{abstract}

\section{Introduction}
Let $X$ be a smooth connected  complex projective variety. If ${\rm
CH}_0(X)=\mathbb{Z}$ (or equivalently the subgroup ${\rm
CH}_0(X)_0\subset {\rm CH}_0(X)$ of $0$-cycles of degree $0$ is
$0$), we have the Bloch-Srinivas decomposition of the diagonal (see
\cite{blochsrinivas}) which says that for some integer $N\not=0$,
\begin{eqnarray}\label{decompintro}
N\Delta_X=Z_1+Z_2\,\,{\rm in}\,\, {\rm CH}^n(X\times X),\,n={\rm
dim}\,X,
\end{eqnarray} where $Z_2=N(X\times x)$ for some point $x\in X(\mathbb{C})$ and $Z_1$ is supported on
$D\times X$, for some proper closed algebraic subset
$D\varsubsetneqq X$. Note  that conversely, if we have a decomposition as in (\ref{decompintro}),
then ${\rm CH}_0(X)=\mathbb{Z}$ because the left hand side acts as
$N$ times the identity on ${\rm CH}_0(X)_{0}$
while the right hand side acts as $0$ on ${\rm CH}_0(X)_{0}$, so that
$N \,{\rm CH}_0(X)_{0}=0$ and finally ${\rm CH}_0(X)_{0}=0$ by Roitman's theorem
\cite{roitor} on torsion.

Denoting $K=\mathbb{C}(X)$, (\ref{decompintro})
is equivalent by the localization exact sequence to the fact  that
the diagonal point $\delta_K$ of $X_K$, that is, the restriction to
${\rm Spec}(K)\times X$ of $\Delta_X$, satisfies
$$N\delta_K=N x_K\,\,{\rm in}\,\,{\rm CH}^n(X_K)={\rm CH}_0(X_K).$$
When there exists a decomposition as in (\ref{decompintro}) with
$N=1$, we will say that $X$ admits a {\it Chow theoretic decomposition of
the diagonal}.  As explained in \cite[Lemma 1.3]{auelCT}, this is
equivalent to saying that ${\rm CH}_0(X)$ is universally trivial, in
the sense that for any field $L$ containing $\mathbb{C}$, ${\rm
CH}_0(X_L)_0=0$.

By taking cohomology classes in (\ref{decompintro}), one gets as well
a decomposition

\begin{eqnarray}\label{decompintrocohN}
N[\Delta_X]=[Z_1]+[Z_2]\,\,{\rm in}\,\, {\rm H}^{2n}(X\times
X,\mathbb{Z}),\,n={\rm dim}\,X,
\end{eqnarray} where $Z_1,\,Z_2, \,D$ are as above, from which Bloch and Srinivas
\cite{blochsrinivas} deduce a number of interesting consequences.
When there exists a decomposition as in (\ref{decompintrocohN}) with
$N=1$, that is
\begin{eqnarray}\label{decompintrocoh}
[\Delta_X]=[Z_1]+[X\times x]\,\,{\rm in}\,\, { H}^{2n}(X\times
X,\mathbb{Z}),
\end{eqnarray} where  $Z_1$ is supported on
$D\times X$, for some proper closed algebraic subset $D\subset X$,
we will say that $X$ admits a {\it cohomological decomposition of
the diagonal}. We started the study of this property in \cite{voisinJAG}, mainly in the case of rationally connected threefolds.

In either of its forms (Chow-theoretic or cohomological), the
existence of a decomposition of the diagonal is an important
criterion for rationality, as noticed in \cite{auelCT},
\cite{voisinJAG}. This property is in fact invariant under stable
birational equivalence, where we say that $X$ and $Y$ are stably
birational if $X\times \mathbb{P}^r\stackrel{{\rm birat}}{\cong}
Y\times \mathbb{P}^s$ for some integers $r$ and $s$. As this
property is obviously satisfied by projective space, an $X$ not
satisfying this criterion is not stably rational.

We  compare
 these two properties  (the existence of a Chow-theoretic and cohomological decompositions of the diagonal)
 in
\cite{voisincubic}, showing in particular that they are equivalent
for cubic fourfolds and odd dimensional smooth cubic hypersurfaces.

It is a basic fact, proved in \cite{voisinJAG}, that a smooth
complex  projective variety $X$ with nontrivial Artin-Mumford
invariant (the torsion subgroup of $H^3(X,\mathbb{Z})$)  does not
admit a cohomological decomposition (and a fortiori a Chow-theoretic
decomposition) of the diagonal. This is the starting point of this
paper and combined with a degeneration (or specialization) argument
(see Theorem \ref{theodeg}), this will provide us with simple
examples of smooth projective unirational threefolds with trivial
Artin-Mumford invariant which do not admit a decomposition of the
diagonal, hence are not stably rational. Note that these varieties
also have trivial higher degree unramified cohomology groups by
\cite{CTvoisin}.

Let us now discuss a little more precisely the results in this paper, distinguishing the Chow-theoretic
and cohomological cases.
\subsection{Chow-theoretic decomposition of the diagonal and application to stable rationality
\label{sec01}} The first examples of unirational threefolds with
nontrivial torsion in their degree $3$ integral Betti cohomology
were constructed by Artin and Mumford \cite{artinmumford}. These
varieties are desingularizations of certain special $10$-nodal
quartic double solids (that is, double covers
$Y\rightarrow\mathbb{P}^3$ ramified along a quartic surface with
$10$ nodes in special position). The examples we consider are smooth
quartic double solids, for which the quartic surface is smooth, or
desingularizations of $k$-nodal quartic double solids, for which the
quartic  surface has $k$ ordinary quadratic singularities, but
unlike the Artin-Mumford examples, the nodes of the quartic surface
will be in general position. They are unirational (see
\cite{beauville}), a fact which in the nodal case is immediate to
prove  by considering the preimages in $Y$ of lines in $
\mathbb{P}^3$ passing through the node : these curves are rational
and the family of these curves endows $Y$ with a conic bundle
structure, which furthermore contains a multisection which is a
rational surface, giving unirationality. If the quartic surface has
at most seven nodes in general position, then $X$ has no torsion in
its integral cohomology. We will prove by a degeneration argument
(cf. Theorem \ref{theodeg},(i)) the following result.
\begin{theo} \label{theointrochow} (Cf. Theorem \ref{theocoro}) Let $X$ be the desingularization of a very general quartic double
solid with at most seven nodes. Then $X$ does not admit an integral
Chow-theoretic decomposition of the diagonal, hence it is not stably rational.

\end{theo}
\begin{rema}\label{rematartive}{\rm The family of $6$ or $7$-nodal quartic double solids, or equivalently $6$ or $7$-nodal quartic
surfaces in $\mathbb{P}^3$,
 is not irreducible (see \cite{endrass}).
Observe however that there is an unique irreducible  component $M$
of the space $M'$ of nodal quartic surfaces in $\mathbb{P}^3$ with
$k\leq7$ nodes dominating $(\mathbb{P}^3)^{(k)}$ by the map
$\phi:M\rightarrow (\mathbb{P}^3)^{(k)}$ which to a $k$-nodal
quartic associates its set of nodes. This  is what we mean in the
Theorem above by ``general quartic double solid $X$ with $k\leq 7$
nodes''. Indeed, given a set $z$ of $k\leq7$ points in general
position, the set $M_z$ of quartics which are singular along $z$ is
a projective space. One easily checks that a generic point in this
projective space parameterizes a nodal quartic surface with exactly
$k$ nodes. Thus if $U$ is  the Zariski open set of
$(\mathbb{P}^3)^{(k)}$ consisting of points $z$ such that the
dimension of $M_z$ is minimal (in fact it is easily proved that this
dimension is the expected dimension $34-4k$), one finds that
$\phi^{-1}(U)$ is a Zariski open set in a projective bundle over
$U$, hence is irreducible. In fact, we also get results  for quartic
double solids with $8$ or $9$ nodes, but their parameter spaces
 are reducible and we do not know to which component of their parameter spaces our results
apply.}
\end{rema}
The case of seven nodes is particularly interesting, because the
intermediate Jacobian of $X$ in this case is of dimension $3$, hence
is as a principally polarized abelian variety  isomorphic to the
Jacobian of a curve. Thus the Clemens-Griffiths criterion for
rationality is satisfied in this case, and as a consequence, even
the irrationality of the very general $7$-nodal  double  solid
proved above was unknown. In the case of $k\leq 6$, the
irrationality of the very general $k$-nodal quartic double solid was
already known using the Clemens-Griffiths criterion, (this is proved
\cite{voisindoubleduke} for all smooth quartic double solids and we
refer to  \cite{cheltsov} for the nodal case,) but  its stable
irrationality was unknown.
\begin{rema}{\rm De Fernex and Fusi \cite{defernex} prove that in dimension $3$, rationality is stable under
specialization to nodal fibers. Thus, if rationality instead of stable rationality is considered,
the above statement concerning the irrationality of the very  general $k$-nodal double solid with
$k\leq7$ could be obtained by applying their argument. Note however that our degeneration theorem
(Theorem \ref{theodeg})
works in any dimension, while theirs works only in dimension $3$.
}
\end{rema}
\subsection{Cohomological decomposition of the diagonal and cycle-theoretic applications}
We wish to describe in more detail in this paragraph which kind of
information can be extracted from the non-existence of a
cohomological decomposition of the diagonal, as this provides more
explicit obstructions to stable rationality. A general reason why it
should be more restrictive than the rational decomposition
(\ref{decompintrocohN}) is the following : If the positive degree
cohomology $H^{*>0}(X,\mathbb{Q})$ has geometric coniveau $\geq1$,
that is, vanishes on the complement of a divisor $D\subset X$, a
property which is implied by the decomposition
(\ref{decompintrocohN}), the same is true for the positive degree
{\it integral}  cohomology $H^{*>0}(X,\mathbb{Z})$ (possibly for a
different divisor $D'$). Indeed, as noted in \cite{CTvoisin}, the
Bloch-Ogus  sheaves $\mathcal{H}^i(\mathbb{Z})$ on $X_{Zar}$
associated to the presheaves $U\mapsto H^i(U,\mathbb{Z})$ have no
$\mathbb{Z}$-torsion, as a consequence of the Bloch-Kato conjecture
proved by Voevodsky \cite{voe}. It follows from this that the class
$[\Delta_X]-[X\times x]$ vanishes on the Zariski open set $U\times
X$ of $X$, where $U:=X\setminus D'$.  This shows that assuming
(\ref{decompintrocohN}), one gets a decomposition
$$[\Delta_X]=\alpha+[X\times x]\,\,{\rm in}\,\, { H}^{2n}(X\times
X,\mathbb{Z})$$ where $\alpha$ is an integral cohomology class
supported on $D'\times X$. The decomposition (\ref{decompintrocoh})
is a much stronger statement  since it asks that the class $\alpha$
is the class of an algebraic cycle supported on $D'$.

The following example on the other hand shows that in the absence of
torsion in cohomology, we  need some transcendental cohomology in
order to prove  the non-existence of an integral cohomological
decomposition of the diagonal :
\begin{example} If $X$ is a smooth projective
variety such that $H^*(X,\mathbb{Z})$ has no torsion and is
algebraic, that is, is generated by classes of algebraic
subvarieties, then $X$ admits an integral cohomological
decomposition of the diagonal. This follows from the K\"unneth
decomposition which expresses the Betti cohomology class of the
diagonal as a $\mathbb{Z}$-linear combination of classes
$pr_1^*\alpha\smile pr_2^*\beta$, and from the fact that all
cohomology classes on $X$ are algebraic.

\end{example}

In the presence of non-trivial odd degree cohomology   (necessarily
of degree $\geq3$ if we already have a rational decomposition as in
(\ref{decompintrocohN})), the integral decomposition
(\ref{decompintrocoh}) becomes restrictive, even when the integral
cohomology is torsion free, as shows our theorem
\ref{theointrocoh30jui}. The obstruction we find is a sort of
secondary obstruction which we describe now. Recall that if
$H^3(X,\mathcal{O}_X)=0$, the intermediate Jacobian of $X$, defined
as the complex torus
\begin{eqnarray}
\label{defJX} J^3(X)=H^3(X,\mathbb{C})/(H^3(X,\mathbb{Z})\oplus
F^2H^3(X,\mathbb{C})),
\end{eqnarray}
is an abelian variety. Furthermore the generalized Hodge conjecture
\cite{groth} predicts that the Griffiths Abel-Jacobi map
$$AJ_X:{\rm CH}^2(X)_{hom}\rightarrow J^3(X)$$
is surjective. When  ${\rm CH}_0(X)=\mathbb{Z}$, this is proved by
Bloch and Srinivas \cite{blochsrinivas} who even  prove under this
assumption the much stronger result  that  $AJ_X:{\rm
CH}^2(X)_{hom}\rightarrow J^3(X)$ is  an isomorphism (see also
\cite{murre}). The situation is then similar to the case of
divisors, where we have the Abel-Jacobi isomorphism $AJ_X:{\rm
Pic}^0(X)\cong J^1(X)$ and where it is well-known  that there exists
a universal (or Poincar\'e) divisor $\mathcal{P}$ on $J^1(X)\times
X$  with the property that the morphism $\phi_{\mathcal{P}}:
J^1(X)\rightarrow J^1(X)$, which to $l\in J^1(X)$ associates
$AJ_X(\mathcal{D}_l)$, is the identity. Our second main result in
this paper is that, from this last viewpoint, codimension $\geq2$
cycles actually behave differently  than divisors, and as we will
see, this is related to the cohomological decomposition of the
diagonal. Let us make the following definition:
\begin{Defi} Let $X$ be  a smooth projective variety  such that
  $AJ_X: {\rm
CH}^2(X)_{hom}\rightarrow J^3(X)$ is an isomorphism (note that
$J^3(X)$ is then automatically an abelian variety).  We will say
that $X$  admits a universal codimension $2$ cycle if there exists a
codimension $2$ cycle $Z\in {\rm CH}^2(J^3(X)\times X)$ such that
$Z_{\mid a\times X}$ is homologous to $0$ for $a\in J^3(X)$ and the
morphism induced by the Abel-Jacobi map
$$\Phi_Z:J^3(X)\rightarrow J^3(X),\,a\mapsto AJ_X(Z_a),$$ is the
identity of $J^3(X)$.
\end{Defi}
\begin{rema}\label{remahodgeclass}
{\rm Note  that under the same assumption  $H^3(X,\mathcal{O}_X)=0$,  there is an integral Hodge class $\alpha$ of
degree $4$ on $J^3(X)\times X$ which is determined up to torsion by
the property that
$$\alpha_*:H_1(J^3(X),\mathbb{Z})\rightarrow
H^3(X,\mathbb{Z})/{\rm torsion}$$ is the canonical
isomorphism given by the definition (\ref{defJX}) of $J^3(X)$ as a complex torus and by the fact that
 $\alpha$ is of type $(1,3)$ in the
K\"unneth decomposition of $H^4(J^3(X)\times X,\mathbb{Z})$.

The rational cohomological decomposition of the diagonal
(\ref{decompintrocohN}) implies that $N\alpha$  (modulo torsion) is
algebraic (see  Lemma \ref{letardif30jui}  and
\cite{blochsrinivas}). So for a smooth complex projective variety
$X$ with  ${\rm CH}_0(X)=\mathbb{Z}$, saying that $X$ admits a
universal codimension $2$ cycle is equivalent to saying that this
degree $4$ {\it integral} Hodge class on $J^3(X)\times X$, which is
known to be $\mathbb{Q}$-algebraic, is actually algebraic (modulo
the torsion of $H^4(J^3(X)\times X,\mathbb{Z})$).}
\end{rema}

 We proved in \cite{voisinJAG} that if $X$ has dimension $3$, the
 integral cohomological decomposition
(\ref{decompintrocoh}) has the following consequences (consequence 1
is due to Bloch and Srinivas \cite{blochsrinivas} and  uses only the
rational cohomological decomposition (\ref{decompintrocohN})):

\begin{theo}  \label{theocitedeJAG} Assume a smooth projective
$3$-fold has an integral cohomological decomposition of the
diagonal. Then

\begin{enumerate}
\item \label{item1} $H^i(X,\mathcal{O}_X)=0$ for $i>0$.

\item \label{item2}   $H^*(X,\mathbb{Z})$ has no torsion.

\item \label{item3}   The even degree integral  cohomology of $X$ consists of classes
of algebraic cycles.

\item \label{item4}  There exists a universal codimension $2$ cycle $Z\in
{\rm CH}^2(J^3(X)\times X)$.
\end{enumerate}
Conversely, if \ref{item1} to \ref{item4} hold and furthermore the
following property holds:

\vspace{0.4cm}

 5. $J^3(X)$ has a $1$-cycle $z\in {\rm CH}^{g-1}(J^3(X))$
 whose cohomology class $[z]\in H^{2g-2}(J^3(X),\mathbb{Z})$ is the minimal class
$\frac{\theta^{g-1}}{(g-1)!}$, where $g:={\rm dim}\,J^3(X)$, and
$\theta\in H^2(J^3(X),\mathbb{Z})$ is the class of the natural
principal polarization on $J^3(X)$,

\vspace{0.4cm}

 then $X$
admits an integral cohomological decomposition of the diagonal.

\end{theo}
\begin{rema}{\rm More recently, we proved in \cite{voisincubic} that Property 5
above is also implied by the existence of a cohomological
decomposition of the diagonal, but we will not need this result
here. }
\end{rema}
 If we consider a rationally connected threefold $X$,
\ref{item1} and \ref{item3} are always satisfied (\ref{item1} is
standard and \ref{item3} is proved in \cite{voisinintegralhodge}).
The  Artin-Mumford example \cite{artinmumford} shows that
\ref{item2} does not always hold even for unirational threefolds and
this provides by the theorem above an obstruction to the existence
of an integral cohomological decomposition of the diagonal.

Property \ref{item4} is much harder to analyze and it is still
unknown if it is satisfied for the general cubic threefold (this
result is claimed in \cite{xuze} but  the proof is incorrect, as
there is a missing term in  the formula used for ${\rm
ch}(\mathcal{O}_X(1))$ in the proof of Theorem 2.3 of {\it loc.
cit.}). We prove in \cite{voisincubic} that it is satisfied by
``many'' smooth cubic threefolds, (more precisely, there is a
countable union of subvarieties of codimension $\leq 3$ in the
moduli space of cubic threefolds parameterizing cubic threefolds
satisfying property \ref{item4}). One of our main results in this
paper is the fact that Property \ref{item4} can fail for some
unirational threefolds. In fact, we will exhibit unirational
threefolds $X$ for which property \ref{item2} (the Artin-Mumford
criterion) and property 5  of Theorem \ref{theocitedeJAG} are
satisfied, but not admitting an integral cohomological decomposition
of the diagonal.
 Such a threefold $X$ does not
admit a universal codimension $2$ cycle by Theorem
\ref{theocitedeJAG}.

\begin{theo}\label{theointrocoh30jui} (Cf. Theorem \ref{theocoro}
and Theorem \ref{theotrue}.)  (i) Let $X$ be the desingularization
of a very general quartic double solid with at most seven nodes.
Then $X$ does not admit an integral cohomological decomposition of
the diagonal.

(ii) If $X$ is the desingularization of a very general quartic
double solid with exactly seven nodes, then $X$ does not admit a
universal codimension $2$ cycle $Z\in {\rm CH}^2(J^3(X)\times X)$.

\end{theo}
 Concerning property 5, that is, the question
whether $J^3(X)$ has a $1$-cycle in the minimal class
$\frac{\theta^{g-1}}{(g-1)!}$, this is a very classical completely
open question for most rationally connected threefolds (in
particular the cubic threefold, for which we prove in
\cite{voisincubic} that a positive answer to this question is
equivalent to the fact that the cubic threefold has universally
trivial ${\rm CH}_0$ group) but also for very general principally
polarized abelian varieties of dimension $\geq4$. The
Clemens-Griffiths criterion for rationality \cite{clemensgriffiths}
states that if a smooth projective threefold is rational, its
intermediate Jacobian is isomorphic as a principally polarized
abelian variety to a product of Jacobians of smooth curves. By the
Matsusaka characterization of products of Jacobians
\cite{matsusaka}, another way to state the Clemens-Griffiths
criterion is to say that there exists an {\it effective} $1$-cycle
in $J^3(X)$ (that is, a combination with positive coefficients of
curves in $J^3(X)$) whose cohomology class is
 the minimal class $\frac{\theta^{g-1}}{(g-1)!}$. This condition   is much more restrictive geometrically
 than Property 5 above. In
particular cases, it can be solved negatively for a general $X$ by a
dimension count for the number of parameters for Jacobians of
curves, or for a specific $X$ by the precise study of the geometry
of the Theta divisor. This cannot be done with the question of the
existence of a $1$-cycle in the minimal class  that is not
necessarily effective. We will not be concerned by this problem
however, since in the examples we will analyze closely, namely the
$7$-nodal quartic double solids, we will have ${\rm dim }\,J^3(X)=
3$ so $J^3(X)$ will automatically satisfy the Clemens-Griffiths
criterion. Together with the triviality of the Artin-Mumford
invariant, this allows us to entirely focus on property \ref{item4},
which by Theorem \ref{theocitedeJAG} above is the only obstruction
to the existence of a cohomological decomposition of the diagonal.

 Our next  result (Theorem
\ref{theounramintro})  will relate the non-existence of a universal codimension
$2$ cycle  to the vanishing of the
{\it universal} third unramified cohomology with torsion coefficients introduced in \cite{auelCT}.
Unramified cohomology with torsion coefficients has been used by
Colliot-Th\'{e}l\`{e}ne and Ojanguren \cite{CToj} as a powerful tool to
detect irrationality (see also \cite{peyre}). In the paper
\cite{auelCT}, the authors  introduce the notion of ``universal
triviality of the third unramified cohomology group of $X$''.  We
just sketch here the idea and refer to \cite{auelCT} for more
details. The universal triviality of the third unramified cohomology
group of $X$ with torsion coefficients is equivalent to the fact
that for any smooth quasi-projective variety $U$ and any class
$\alpha\in H^3_{nr}(U\times X,\mathbb{Q}/\mathbb{Z})$, there is a
Zariski dense open set $U'\subset U$ such that $\alpha_{\mid
U'\times X}$ is the pull-back of a class $\beta\in
H^3(U',\mathbb{Q}/\mathbb{Z})$. In the paper \cite{auelCT}, the
notion is more elegantly formulated since the authors can use the
\'etale cohomology of the variety $X_K$ where $K$ is the function
field of $U$, but in the context of Betti cohomology, we have to
formulate it by taking the direct limit over Zariski open sets. In
any case, the notion is obviously particularly interesting for those
varieties with vanishing third unramified cohomology group with
torsion coefficients, as it is the case for rationally connected
threefolds (see \cite{CTvoisin}).

\begin{theo}\label{theounramintro} (Cf. Theorem \ref{theounram}) Let $X$ be a smooth complex projective variety of dimension
$n$ with ${\rm CH}_0(X)=\mathbb{Z}$. Assume
\begin{enumerate}
\item \label{hytorintro} $H^*(X,\mathbb{Z})$ has no torsion and the K\"unneth components of the diagonal are algebraic.
\item \label{hy3intro} The group $H^3_{nr}(X,\mathbb{Q}/\mathbb{Z})$ is trivial (or equivalently by \cite{CTvoisin}, the integral Hodge classes of degree $4$ on $X$ are algebraic).
\end{enumerate}
Then the degree $3$ unramified cohomology of $X$ with torsion
coefficients is universally trivial if and only if there is a
universal codimension $2$ cycle $Z\in {\rm CH}^2(J^3(X)\times X)$.
\end{theo}
 These hypotheses apply to any rationally
connected threefold with no torsion in $H^3$ (see Corollary
\ref{corofififigeneral}), and also to cubic fourfolds, proving in
particular the universal triviality of the third unramified
cohomology group with torsion coefficients of any smooth cubic
fourfold (see Example \ref{exacubic}).

\begin{coro} \label{cornew30juill} If $X$ is as in Theorem \ref{theointrocoh30jui}, (ii), the universal third unramified
cohomology group of $X$ with coefficients in $\mathbb{Q}/\mathbb{Z}$
is not universally trivial.
\end{coro}

\subsection{A geometric application}

We conclude this paper with a result of a more geometric nature. It
concerns the following weaker version (*) of property \ref{item4} of
Theorem \ref{theocitedeJAG},
 originating in work of de Jong and Starr (see \cite{dJscubic4}),
 which
 is very natural if one thinks of  the geometry of the Abel map for
 $0$-cycles on curves:

\vspace{0.5cm}

(*)  {\it There exist a smooth projective variety $B$ and a
codimension $2$ cycle $Z\in {\rm CH}^2(B\times X)$ inducing a
surjective map
$$\Phi_Z:B\rightarrow J^3(X),\,b\mapsto  AJ_X(Z_b)$$
with rationally connected general fibers.}

\vspace{0.5cm}

Property (*)  is  satisfied by cubic threefolds (see
\cite{tikhomarku}). In \cite{voisinJAG}, we observed that (*)
  has already very nice consequences
on the integral Hodge conjecture (for instance, if $X$ is rationally
connected and  $H^3(X,\mathbb{Z})$ has no torsion,  property (*)
implies that the integral Hodge conjecture holds for products
$C\times X$, where $C$ is a smooth curve).

 Coming back to Theorem \ref{theointrocoh30jui}, (ii), we will even prove a stronger
statement (see Theorem \ref{theotrue}), namely that in the situation
and with the notation above, $X$ does not satisfy property (*). This
in particular answers negatively the following question originally
asked
 by  Harris, de Jong and Starr, solved positively
 for  the intersection of two quadrics \cite{castravet} and for many degrees for the cubic
threefold \cite{tikhomarku}, \cite{hrs}, \cite{voisinJAG}:

\vspace{0.5cm}

{\bf Question}: {\it Let $X$ be a rationally connected threefold. Is
it true that the Abel-Jacobi map on the main component of the family
of rational curves of sufficiently positive class has rationally
connected general fiber}?

\vspace{0.5cm}

 (Note that
the paper \cite{dJscubic4} exhibits  a totally different behaviour
for rational curves on cubic fourfolds. Note also that Castravet
\cite{castravet} gives examples of $X$ as above, with
$H_2(X,\mathbb{Z})=\mathbb{Z}$ and for which the family of free
rational curves of degree $d$ is reducible for all degrees $d$, the
general point of the ``main component'' parameterizing  a very free
rational curve.)

The method of the proof of Theorems \ref{theointrochow} and
\ref{theointrocoh30jui}, (i), is by degeneration of the general
quartic double solid $X_t$, (or the general $k$-nodal double solid
$X_t$,) to the Artin-Mumford nodal quartic double solid $X_0$. Our
general result proved in Section \ref{sec1} is the general
specialization theorem \ref{theodeg} implying in our case that the
non-existence of a decomposition of the diagonal (Chow-theoretic or
cohomological) for a desingularization of $X_0$ implies the
non-existence also for the very general smooth double solid $X_t$,
or for the desingularization $\widetilde{X}_t$ of a very general
double solid with $k\leq7$ nodes. The interesting fact in this
degeneration argument is the following: the desingularization of the
Artin-Mumford double solid does not admit an integral cohomological
decomposition of the diagonal because it has some nontrivial torsion
in its integral cohomology. The desingularization of the general
$k\leq7$-nodal double solid then does not admit  an integral
cohomological decomposition of the diagonal, but it has no torsion
anymore in its integral cohomology. The non-existence of integral
cohomological decomposition of the diagonal thus implies that
another property from \ref{item1} to $5$ in Theorem
\ref{theocitedeJAG} must be violated, and when $k=7$, the only one
which can be violated is the existence of a universal codimension
$2$ cycle.

\vspace{0.5cm}

{\bf Thanks.}  I thank  Jean-Louis Colliot-Th\'el\`ene  for sending
me the very interesting paper \cite{auelCT}, for inspiring
discussions related to it and for his criticism on the exposition. I also thank the anonymous referee for his very helpful  criticism and suggestions.

\section{A degeneration argument\label{sec1}}
We prove in this section the following  degeneration (or
specialization) result.
\begin{theo}\label{theodeg} Let $\pi:\mathcal{X}\rightarrow B$ be a flat  projective
morphism of relative dimension $n\geq2$, where  $B$ is a smooth curve \footnote{I thank J.-L.
Colliot-Th\'el\`ene for bringing to my attention the fact that the
smoothness assumption I had originally put on $\mathcal{X}$  was in
fact not used in the proof.}. Assume that the fiber $\mathcal{X}_t$
is smooth for $t\not=0$, and has at worst ordinary quadratic
singularities for $t=0$. Then

(i) If for general $t\in B$, $\mathcal{X}_t$ admits a Chow theoretic
decomposition of the diagonal (equivalently, $CH_0(\mathcal{X}_t)$
is universally trivial), the same is true for any smooth projective
model $\widetilde{\mathcal{X}}_o$ of $\mathcal{X}_o$.

(ii) If for general $t\in B$, $\mathcal{X}_t$ admits a cohomological
decomposition of the diagonal, and the even degree integral homology
of a smooth projective model $\widetilde{\mathcal{X}}_o$ of
$\mathcal{X}_o$ is algebraic (i.e. generated over $\mathbb{Z}$ by
classes of subvarieties),  $\widetilde{\mathcal{X}}_o$ also admits a
cohomological decomposition of the diagonal.
\end{theo}
\begin{rema}{\rm As the proof will show, the assumptions on the
singularities of the central fiber can be weakened as follows: For
(i), it suffices to assume that the central fiber is irreducible and
admits a desingularization $\widetilde{\mathcal{X}}_o\rightarrow
\mathcal{X}_o$
with smooth exceptional divisor $E$, whose connected
components
 $E_i$
have universally trivial ${\rm CH}_0$ group (for example, is rational)\footnote{A. Pirutka and
J.-L. Colliot-Th\'{e}l\`{e}ne \cite{pico} have indeed applied successfully
the same method in the case of a specialization to more complicated
singularities, generalizing our  results to  quartic threefolds,
instead of quartic double solids.}. For (ii), it would suffice to
know also that with the same conditions on the desingularization,
the even degree integral homology of $\widetilde{\mathcal{X}}_o$  is
algebraic, the odd degree cohomology of $E$ is trivial and the even
degree integral homology of $E$ is without torsion and generated by
classes of algebraic cycles. These properties are clearly satisfied
when $E$ is a disjoint union of smooth quadrics.}
\end{rema}

\begin{rema}{\rm Theorem \ref{theodeg} will be used in applications
in the following form: If the desingularization
$\widetilde{\mathcal{X}}_o$ of the central fiber does not admit an
integral Chow-theoretic or cohomological decomposition of the
diagonal, the {\it very general} fiber $\mathcal{X}_t$ does not
admit an integral Chow-theoretic or cohomological decomposition of
the diagonal. This is because the set of points $t\in B_{reg}$ such
that the fiber $\mathcal{X}_t$ is smooth and admits an integral
Chow-theoretic (resp. cohomological) decomposition of the diagonal
is a countable union of closed algebraic subsets of  $B_{reg}$, see
below.}
\end{rema}
Before giving the proof of the theorem, let us make one more remark. The proof is in two steps: The first one consists in proving that the central fiber $\mathcal{X}_o$ admits
a decomposition of the diagonal. This  follows from a specialization result  due to Fulton
\cite{fulton}
 which we will reprove below for completeness. One has to be a little careful here since $\mathcal{X}_o$ is singular. The decomposition of the diagonal for $\mathcal{X}_o$ will hold in ${\rm CH}_n(\mathcal{X}_o\times \mathcal{X}_o)$. This first step does not use any
  assumption on the singularities of $\mathcal{X}_o$. The second step consists in comparing
 what happens for $\mathcal{X}_o$ and $\widetilde{\mathcal{X}}_o$. It is here that the assumption on the singularities is used in an essential way. Note that
 Theorem \ref{theodeg} is not true without assumptions on the singularities of the central fiber, as the simplest example
 shows: Consider a family of smooth cubic surfaces specializing to a cubic surface which is a cone
 over a smooth elliptic curve. Then the general fiber is rational, hence admits a Chow-theoretic decomposition of the diagonal, but the desingularization $\widetilde{\mathcal{X}}_o$ of the central fiber, being a
 $\mathbb{P}^1$-bundle over an elliptic curve $E$, does not admit one, even with $\mathbb{Q}$-coefficients, since ${\rm CH}_0(\widetilde{\mathcal{X}}_o)=
 {\rm CH}_0(E)_0$ is nontrivial.
\begin{proof}[Proof of Theorem \ref{theodeg}] (i) For a general $t\in B$, there exist
an effective divisor $D_t\subset \mathcal{X}_t$
and a cycle $Z_t$ supported on $D_t\times \mathcal{X}_t$ such that
for  any point $x_t\in \mathcal{X}_t$, one has
\begin{eqnarray}
\label{eqt}\Delta_{\mathcal{X}_t}=Z_t+ \mathcal{X}_t\times
x_t\,\,{\rm in}\,\, {\rm CH}^n(\mathcal{X}_t\times \mathcal{X}_t).
\end{eqnarray}
The set of data $(D_t,x_t,Z_t)$ above is parameterized by a
countable union of algebraic varieties over $B$  whose image in $B$
contains a Zariski open set. It follows by a Baire category argument
that one of these algebraic varieties dominates $B$, so that there
is a smooth finite cover $B'\rightarrow B$ and a divisor
$\mathcal{D}\subset \mathcal{X}':=B'\times_B\mathcal{X}$, which we
may assume to contain no fiber of $\mathcal{X}'\rightarrow B'$, a
section $\sigma:B'\rightarrow \mathcal{X}'$ and a codimension $n$ (or dimension $n+1$)
cycle $\mathcal{Z}$ in $\mathcal{X}'\times_{B'}\mathcal{X}'$
supported on $\mathcal{D}\times_{B'} \mathcal{X}'$, with the
properties that for general $t\in B'$, \begin{eqnarray}
\label{eqtrest} \Delta_{\mathcal{X}'_t}=\mathcal{Z}_t+
\mathcal{X}'_t\times \sigma(t)\,\,{\rm in}\,\, {\rm
CH}_n(\mathcal{X}'_t\times \mathcal{X}'_t),
\end{eqnarray}
where $\mathcal{Z}_t=\mathcal{Z}_{\mid \mathcal{X}'_t\times
\mathcal{X}'_t}\in {\rm CH}_n(\mathcal{X}'_t\times
\mathcal{X}'_t)$ is well-defined even if $\mathcal{X}'\times_{B'}
\mathcal{X}'$ is singular, because $\mathcal{X}'_t\times
\mathcal{X}'_t$ is a Cartier divisor in $\mathcal{X}'\times_{B'}
\mathcal{X}'$ (see \cite[2.3]{fulton}). Now we use the following
general  fact (we include a proof for completeness):
\begin{prop} \label{prop31jui} Let $\pi:\mathcal{Y}\rightarrow B$ be a flat  morphism of
algebraic varieties, where $B$ is smooth of dimension $r$, and let $\mathcal{Z}\in
{\rm CH}_N(\mathcal{Y})$ be a cycle. Then the set $B_Z$ of points
$t\in B$ such that $\mathcal{Z}_t:=\mathcal{Z}_{\mid \mathcal{Y}_t}$
vanishes in ${\rm CH}_{N-r}(\mathcal{Y}_t)$ is a countable union of proper
closed algebraic subsets of $B$.
\end{prop}
\begin{proof} For any $t\in B$ such that $\mathcal{Z}$ and $ \mathcal{Y}_t$
intersect properly and $\mathcal{Z}_{\mid \mathcal{Y}_t}=0$ in ${\rm
CH}_{N-r}(\mathcal{Y}_t)$, there exist subvarieties $W_{i,t}\subset
\mathcal{Y}_t$ of dimension $N-r+1$ and nonzero rational functions $\phi_{i,t}$ on
$\widetilde{W_{i,t}}\stackrel{j_{i,t}}{\rightarrow}\mathcal{Y}_t$
such that $\sum_i{j_{i,t}}_*{\rm div}\,\phi_{i,t}=\mathcal{Z}_{\mid
\mathcal{Y}_t}$. The data $(W_{i,t},\phi_{i,t})$ are parameterized
by a countable union of quasi-projective irreducible varieties
$M_l\stackrel{\alpha_l}{\rightarrow} B$, whose image in $B$ is
exactly the set $B_Z$. We may assume that the $M_l$'s are smooth.
For each of these varieties $M_l$, there exists a closed algebraic
subvariety $B_l\subset B$, such that ${\rm Im}\,\alpha_l$ contains a
Zariski open set $B_l^0$ of $B_l$. Let $B_l^0=\alpha_l(M_l^0)$. For
any point $t$ of $B_l^0$, we have $\mathcal{Z}_{\mid
\mathcal{Y}_t}=0$ in ${\rm CH}_{N-r}(\mathcal{Y}_t)$ and it only remains
to show that this remains true for any point of $B_l$, since we then
have $B_Z=\cup_lB_l$. We observe now that we can assume that the
$M_l$'s carry universal objects, and thus that the cycle
$\mathcal{Z}_l:={\alpha_l'}^*\mathcal{Z}$ is rationally equivalent
to $0$ on $\mathcal{Y}'_{M_l}$, where
$$\mathcal{Y}'_{M_l}:=M_l\times_B\mathcal{Y}\stackrel{\pi'}{\rightarrow} M_l,$$
 and $\alpha'_l:\mathcal{Y}'_{M_l}\rightarrow \mathcal{Y}$ is the
natural map. Here we use the fact, which will be used again below,
 that the morphism
$${\alpha_l'}^*:{\rm CH}_N(\mathcal{Y})\rightarrow {\rm CH}_{N'}(\mathcal{Y}'_{M_l}),\,N'=N-r+ {\rm dim}\,M_l$$
is well-defined, because $B$ is smooth and $\mathcal{Y}\rightarrow B$ is flat.
Let now $\overline{M_l}$ be a smooth partial completion
of $M_l$ on which the morphism $\alpha_l$ extends to a projective
morphism $\overline{\alpha}_l:\overline{M_l}\rightarrow B$ whose
image is equal to $B_l$ by properness. Denote by
$\overline{\mathcal{Y}'_{M_l}}$ the fibered product
$\overline{M_l}\times_B\mathcal{Y}$, with natural morphism
$\pi'':\overline{\mathcal{Y}'_{M_l}}\rightarrow \overline{M_l}$ extending
$\pi'$, and by
$\overline{\alpha'_l}:\overline{\mathcal{Y}'_{M_l}}\rightarrow
\mathcal{Y}$ the natural map.

 Then ${\rm
Im}\,\overline{\alpha}_l=B_l$, and for any $s\in \overline{M_l}$
with image $t\in B_l$, we have
$${{\overline{\alpha}'_l}^*(\mathcal{Z})_{\mid\overline{\mathcal{Y}'_{M_l}}}}
_s=\mathcal{Z}_{\mid
\mathcal{Y}_t}\,\,{\rm in}\,\,{\rm CH}_{N-r}(\mathcal{Y}_t).$$ As  the
cycle ${\overline{\alpha}'_l}^*(\mathcal{Z})$ vanishes on the
Zariski open set ${\pi''}^{-1}(M_l^0)={\pi'}^{-1}(M_l^0)$, we thus conclude applying to
$M=\overline{M_l}$, $W=\overline{\mathcal{Y}'_{M_l}}$, $f=\pi''$,
$\Gamma={\overline{\alpha}'_l}^*(\mathcal{Z})$ the following lemma:

\begin{lemm}  Let $M$ be smooth of dimension $m$ and let  $f:W\rightarrow M$ be a flat morphism. Let $\Gamma$ be a
$N$-cycle on $W$. Assume there is a Zariski dense open set $M^0$ of $M$
such that $\Gamma_{\mid W^0}=0$ in ${\rm CH}(W^0)$, where
$W^0:=f^{-1}(M^0)$. Then for any $t\in M$, $\Gamma_{\mid W_t}=0$ in
${\rm CH}_{N-m}(W_t)$.
\end{lemm}
\begin{proof} We may assume that $M$ is a smooth curve.
 Let $D$ be the divisor
$M\setminus M^0\subset   M$. By the localization exact sequence, there is a $N$-cycle $z$ supported on $f^{-1}(D)$ such that
$$\Gamma=i_{*}(z)\,\,{\rm in}\,\,{\rm CH}_N(W),$$
where $i$ is the inclusion of the Cartier divisor $f^{-1}(D)$ in $W$.
The fact that $\Gamma_{\mid W_t}=0$ in
${\rm CH}_{N-m}(W_t)$  for any $t\in D$ then follows from the fact that
$f^{-1}(D)_{\mid W_t}$ is the trivial Cartier divisor on $W_t$, and from Fulton's definition of the intersection
with a Cartier divisor, which says that
for any point of $t\in D$, $i_{*}(z)_{\mid W_t}=f^{-1}(D)_{\mid W_t}\cdot z_t$, where
$z_t$ is the part of $z$ lying in the component $W_t$ of $ f^{-1}(D)$.
\end{proof}
It follows that for any $t\in\overline{M_l}$, the restriction of the
cycle ${\overline{\alpha}'_l}^*(\mathcal{Z})$ to the fiber
${\pi'}^{-1}(t)\subset \overline{\mathcal{Y}'_{M_l}}$ is trivial,
and this implies that  for any $t\in B_l$, the restriction of the
cycle $\mathcal{Z}$ to the fiber ${\pi}^{-1}(t)\subset
\mathcal{Y}_t$ is trivial.

\end{proof}
 By Proposition \ref{prop31jui}, the locus of points $t\in B'$ such that
(\ref{eqtrest}) holds is  a countable union of closed algebraic
subsets of $B'$.
 As it contains  by assumption
a Zariski open set, we conclude that (\ref{eqtrest}) holds for any
$t\in B'$. Choose for $t$ any point $o'$ over $o\in B$. Then
identifying $\mathcal{X}'_{o'}$ with $\mathcal{X}_o$ we conclude
that there exist a divisor $D_o\subset \mathcal{X}_o$, a point
$x_o\in\mathcal{X}_o$ and a cycle $Z_o$ supported on $D_o\times
\mathcal{X}_o$, such that
\begin{eqnarray}
\label{eqto} \Delta_{\mathcal{X}_o}={Z}_o+ \mathcal{X}_o\times
x_o\,\,{\rm in}\,\, {\rm CH}_n(\mathcal{X}_o\times \mathcal{X}_o).
\end{eqnarray}

Let $\tau:\widetilde{\mathcal{X}}_o\rightarrow \mathcal{X}_o$ be the
desingularization obtained by blowing-up the singular points. It
remains to deduce from  (\ref{eqto}) that
$\widetilde{\mathcal{X}}_o$ satisfies the same property. Let
$x_i,\,i=1,\ldots,\,N$ be the singular points of $\mathcal{X}_o$,
and $E_i:=\tau^{-1}(x_i)$. By assumption, $E_i$ is a smooth quadric
of dimension $\geq1$, and in particular $E_i$ is rational and has
universally trivial ${\rm CH}_0$ group. Let $E:=\cup_iE_i$. Then
(\ref{eqto}) restricted to
$$(\mathcal{X}_o\setminus\{x_1,\ldots,x_N\})\times (\mathcal{X}_o \setminus\{x_1,\ldots,x_N\})\cong
(\widetilde{\mathcal{X}}_o\setminus E)\times
(\widetilde{\mathcal{X}}_o \setminus E)$$ provides:
\begin{eqnarray}
\label{eqtotilde-E} \Delta_{\widetilde{\mathcal{X}}_o\setminus
E}={Z}_o+ (\widetilde{\mathcal{X}}_o\setminus E)\times x_o\,\,{\rm
in}\,\, {\rm CH}_n((\widetilde{\mathcal{X}}_o\setminus E)\times
(\widetilde{\mathcal{X}}_o \setminus E)).
\end{eqnarray}
where $Z_o$ is supported on $D_o\times
(\widetilde{\mathcal{X}}_o\setminus E)$ for a  proper closed
algebraic subset
 $D_o$ of $\widetilde{\mathcal{X}}_o\setminus E$.
 The localization exact sequence
 allows to rewrite (\ref{eqtotilde-E}) as
 \begin{eqnarray}
 \label{eqtotilde}
\Delta_{\widetilde{\mathcal{X}}_o}={Z}_o+ \widetilde{\mathcal{X}}_o
\times x_o+Z\,\,{\rm in}\,\, {\rm
CH}_n(\widetilde{\mathcal{X}}_o\times \widetilde{\mathcal{X}}_o).
\end{eqnarray}
where $Z_o$ is supported on $D'_o\times \widetilde{\mathcal{X}}_o$
for a proper closed algebraic subset
 $D'_o\subsetneqq \widetilde{\mathcal{X}}_o$ and $Z$ is supported on
  $\widetilde{\mathcal{X}}_o\times E\cup E\times \widetilde{\mathcal{X}}_o$.
 Writing $Z$ as $Z_1+Z_2$, where $Z_1$ is supported on $ E\times \widetilde{\mathcal{X}}_o$, and
 $Z_2$ is supported on $ \widetilde{\mathcal{X}}_o\times E $,
 it is clear that up to replacing in (\ref{eqtotilde})
 $Z_o$ by $Z_o+Z_1$ and  $D'_o$ by $D'_o\cup E$, we may assume that $Z=Z_2$ is supported on
 $$\widetilde{\mathcal{X}}_o\times E=\bigsqcup_i \widetilde{\mathcal{X}}_o\times E_i.$$
 Let $Z_{2,i}$ be the restriction of $Z_2$ to the connected  component
$ \widetilde{\mathcal{X}}_o\times E_i$.
 As ${\rm dim}\,Z_2=n={\rm dim}\,\widetilde{\mathcal{X}}_o$ and $E_i$ has universally
  trivial ${\rm CH}_0$ group, one can write for each $i$
   \begin{eqnarray}
 \label{eqZi}
 Z_{2,i}=Z_{2,i}'+\mu_i \widetilde{\mathcal{X}}_o\times x_i
  \end{eqnarray}
  in ${\rm CH}_n(\widetilde{\mathcal{X}}_o\times E_i)$, where
  $x_i$ is a point of $E_i$, $\mu_i\in\mathbb{Z}$ and
 $ Z_{2,i}'$ is supported on $D_i\times E_i$ for some proper closed algebraic subset $D_i$ of
 $\widetilde{\mathcal{X}}_o$.

 Combining (\ref{eqtotilde}) and (\ref{eqZi}), we finally conclude that
  \begin{eqnarray}\label{eqfin} \Delta_{\widetilde{\mathcal{X}}_o}={Z}'_o+ \widetilde{\mathcal{X}}_o \times x_o+
 \sum_i\mu_i \widetilde{\mathcal{X}}_o\times x_i\,\,{\rm in}\,\, {\rm CH}^n(\widetilde{\mathcal{X}}_o\times \widetilde{\mathcal{X}}_o) ,
 \end{eqnarray}
 where  $Z'_o=Z_o+Z_1+\sum_{i}Z'_{2,i}$ is supported on $ D''_o\times \widetilde{\mathcal{X}}_o$ for
 a proper closed  algebraic subset
 $D''_o=D'_o\cup E\cup (\cup_i D_i)$ of $\widetilde{\mathcal{X}}_o$.
 Formula (\ref{eqfin}) implies that
 $\sum_i\mu_i=0$ and that  ${\rm CH}_0( \widetilde{\mathcal{X}}_o)$ is generated by
 the $x_i,\,i=o,\,1,\ldots , N$, which easily implies that ${\rm CH}_0( \widetilde{\mathcal{X}}_o)=\mathbb{Z}$
 since the central fiber is irreducible under our assumptions (this is why we impose
 the condition that the fiber dimension is $\geq2$; in fact, Theorem \ref{theodeg} is wrong if the fiber dimension is $1$, because the
 disjoint union of two $\mathbb{P}^1$ does not admit a Chow-theoretic decomposition of the diagonal).
 Hence (\ref{eqfin}) gives
 $$\Delta_{\widetilde{\mathcal{X}}_o}={Z}'_o+ \widetilde{\mathcal{X}}_o \times x_o\,\,{\rm in}\,\, {\rm CH}^n(\widetilde{\mathcal{X}}_o\times \widetilde{\mathcal{X}}_o) ,$$
 which concludes the proof of (i).

\vspace{0.5cm}

(ii) The proof of (ii) works very similarly. We first  construct as
before a smooth finite cover $B'\rightarrow B$, a divisor
$\mathcal{D}\subset \mathcal{X}':=B'\times_B\mathcal{X}$, which we
may assume to contain no fiber of $\mathcal{X}'\rightarrow B'$, a
section $\sigma:B'\rightarrow \mathcal{X}'$ and a codimension $n$
cycle $\mathcal{Z}$ in $\mathcal{X}'\times_{B'}\mathcal{X}'$
supported on $\mathcal{D}\times_{B'} \mathcal{X}'$, with the
properties that for general $t\in B'$ (so in particular
$\mathcal{X}'_t$ is smooth), we have the equality of cycle classes
 \begin{eqnarray}
 \label{eqcycle1}
 [\Delta_{\mathcal{X}'_t}]=[\mathcal{Z}_t]+
 [\mathcal{X}'_t\times \sigma(t)]\,\,{\rm in}\,\, { H}^{2n}(\mathcal{X}'_t\times \mathcal{X}'_t,\mathbb{Z})=
 { H}_{2n}(\mathcal{X}'_t\times \mathcal{X}'_t,\mathbb{Z}),
\end{eqnarray}
where $\mathcal{Z}_t=\mathcal{Z}_{\mid \mathcal{X}'_t\times
\mathcal{X}'_t}$. We now work in the analytic setting and restrict
to $\mathcal{X}'_{\Delta}$, where
 $\Delta$ is a small disc in  $B'$ centered at  a point $o'$ of $B'$ over $o\in B$, such
 that $\mathcal{X}'_{\Delta}$   retracts continuously for the usual topology  on the
 central fiber $\mathcal{X}'_{o'}$, (the retraction map being homotopic over $\Delta$ to the
 identity,) so that  $\mathcal{X}'_{\Delta}\times_{\Delta}\mathcal{X}'_{\Delta}$ retracts
 similarly  on
$\mathcal{X}'_{o'}\times \mathcal{X}'_{o'}\cong
\mathcal{X}_{o}\times \mathcal{X}_{o} $. Then we conclude that
(\ref{eqcycle1}) implies

\begin{eqnarray}
\label{eqtohom} [\Delta_{\mathcal{X}_o}]=[\mathcal{Z}_o]+
[\mathcal{X}_o\times x_o]\,\,{\rm in}\,\,
{H}_{2n}(\mathcal{X}_o\times \mathcal{X}_o,\mathbb{Z}),
\end{eqnarray}
where the cycle classes are from now on taken in homology. This is
because the topological retraction from
$\mathcal{X}'_{\Delta}\times_{\Delta}\mathcal{X}'_{\Delta}$ to
$\mathcal{X}'_{o'}\times \mathcal{X}'_{o'}$ induces an isomorphism
$$H_*(\mathcal{X}'_{o'}\times \mathcal{X}'_{o'},\mathbb{Z})\cong
H_*(\mathcal{X}'_{\Delta}\times_{\Delta}\mathcal{X}'_{\Delta},\mathbb{Z})$$
and by flatness, the classes
$$[\Delta_{\mathcal{X}'_{t}}],\,[\mathcal{Z}_t]\in
H_{2n}(\mathcal{X}'_{\Delta}\times_{\Delta}\mathcal{X}'_{\Delta},\mathbb{Z})$$
are constant (that is, independent of $t$) in
$$H_{2n}(\mathcal{X}'_{\Delta}\times_{\Delta}\mathcal{X}'_{\Delta},\mathbb{Z})=H_{2n}(\mathcal{X}'_{o'}\times
\mathcal{X}'_{o'},\mathbb{Z}).$$

With the same notation $\widetilde{\mathcal{X}}_o,x_i,\,E_i,\,E$ as
in the proof of (i), we deduce from (\ref{eqtohom}) by taking the
image in the relative  homology of the pair $({\mathcal{X}}_o\times
{\mathcal{X}}_o,({\mathcal{X}}_o\times \{x_1,\ldots x_N\}\cup
\{x_1,\ldots x_N\}\times {\mathcal{X}}_o))$ the following equality
in the relative homology group
$$H_{2n}({\mathcal{X}}_o\times {\mathcal{X}}_o,({\mathcal{X}}_o\times \{x_1,\ldots x_N\}\cup \{x_1,\ldots x_N\}
\times
{\mathcal{X}}_o),\mathbb{Z})=H_{2n}(\widetilde{\mathcal{X}}_o\times
\widetilde{\mathcal{X}}_o,(\widetilde{\mathcal{X}}_o\times E\cup
E\times \widetilde{\mathcal{X}}_o),\mathbb{Z})$$

\begin{eqnarray}
\label{eqtohomrel}
[\Delta_{\widetilde{\mathcal{X}}_o}]_{rel}=[\mathcal{Z}_o]_{rel}+
[\mathcal{X}_o\times x_o]_{rel}\,\,{\rm in}\,\,
{H}_{2n}(\widetilde{\mathcal{X}}_o\times
\widetilde{\mathcal{X}}_o,(\widetilde{\mathcal{X}}_o\times E\cup
E\times \widetilde{\mathcal{X}}_o),\mathbb{Z}),
\end{eqnarray}
where the subscript ``rel'' indicates that we consider the homology
class in the relative homology group
${H}_{2n}(\widetilde{\mathcal{X}}_o\times
\widetilde{\mathcal{X}}_o,(\widetilde{\mathcal{X}}_o\times E\cup
E\times \widetilde{\mathcal{X}}_o),\mathbb{Z})$. Formula
(\ref{eqtohomrel}) and the long exact sequence of relative homology
imply that the homology  class
$$[\Delta_{\widetilde{\mathcal{X}}_o}]-[\mathcal{Z}_o]- [\widetilde{\mathcal{X}}_o\times x_o]\in {H}_{2n}(\widetilde{\mathcal{X}}_o\times \widetilde{\mathcal{X}}_o,\mathbb{Z})$$
comes from a homology class
$$\beta\in H_{2n}(\widetilde{\mathcal{X}}_o\times E\cup E\times \widetilde{\mathcal{X}}_o,\mathbb{Z}).$$
Note now that the closed subset
$$\widetilde{\mathcal{X}}_o\times E\cup E\times \widetilde{\mathcal{X}}_o\subset \widetilde{\mathcal{X}}_o\times \widetilde{\mathcal{X}}_o$$
is the union of $\widetilde{\mathcal{X}}_o\times E$ and
$E\times\widetilde{\mathcal{X}}_o $ glued along $E\times E$. We thus
have a Mayer-Vietoris exact sequence
$$  \ldots H_{2n}(\widetilde{\mathcal{X}}_o\times E,\mathbb{Z})\oplus
 H_{2n}( E\times \widetilde{\mathcal{X}}_o,\mathbb{Z})
 \rightarrow H_{2n}(\widetilde{\mathcal{X}}_o\times E\cup E\times \widetilde{\mathcal{X}}_o,\mathbb{Z})
\rightarrow H_{2n-1}(E\times E,\mathbb{Z})\rightarrow \ldots$$ As
$E\times E=\bigsqcup_{i,j} E_i\times E_j$ and $E_i\times E_j$ has
trivial homology in odd degree, we conclude that $H_{2n-1}(E\times
E,\mathbb{Z})=0$, so that  $\beta$ comes from a homology class
$$\gamma=(\gamma_1, \gamma_2)\in H_{2n}(\widetilde{\mathcal{X}}_o\times E,\mathbb{Z})\oplus H_{2n}( E\times \widetilde{\mathcal{X}}_o,\mathbb{Z})=H^{2n-2}(\widetilde{\mathcal{X}}_o\times E,\mathbb{Z})\oplus H^{2n-2}( E\times \widetilde{\mathcal{X}}_o,\mathbb{Z})
.$$ We now use the assumption made on $\widetilde{\mathcal{X}}_o$,
namely that its even degree cohomology is algebraic. As the
cohomology of $E$ has no torsion and is algebraic, we get by the
K\"unneth decomposition that
$$H^{2n-2}(E\times \widetilde{\mathcal{X}}_o,\mathbb{Z})=\oplus_{0\leq 2i\leq 2n-2}
H^{2i}(E,\mathbb{Z})\otimes H^{2n-2-2i}(
\widetilde{\mathcal{X}}_o,\mathbb{Z})$$ is generated by  classes of
algebraic cycles $z_{j}\times z'_{j}\subset E\times
\widetilde{\mathcal{X}}_o$ and similarly for
$\widetilde{\mathcal{X}}_o\times E$.

Putting everything together, we get an equality
$$[\Delta_{\widetilde{\mathcal{X}}_o}]-[\mathcal{Z}_o]- [\mathcal{X}_o\times x_o]=\sum_jn_j [z_j\times z'_{j}]+\sum_kn'_k [z'_k\times z_k]\,\,{\rm in}\,\,{H}_{2n}(\widetilde{\mathcal{X}}_o\times \widetilde{\mathcal{X}}_o,\mathbb{Z}).$$
This provides us with an integral  cohomological decomposition of
the diagonal for $\widetilde{\mathcal{X}}_o$ since in the term on
the right, all the cycle classes of the form
$[\widetilde{\mathcal{X}}_o\times point]$ are cohomologous and they
have to sum-up to zero, while all the other terms $[z'_k\times z_k]$
with ${\rm dim}\,z'_k<n$ are supported on $D\times
\widetilde{\mathcal{X}}_o$ for some proper closed algebraic subset
$D$ of $\widetilde{\mathcal{X}}_o$.

\end{proof}

Let us now deduce from Theorem \ref{theodeg} the following result:
\begin{theo} (cf. Theorem \ref{theointrochow} and Theorem \ref{theointrocoh30jui}, (i)) \label{theocoro} Let $\widetilde{X}$ be the natural
 desingularization of  a very general quartic double solid $X$ with $k\leq 7$ nodes.
Then the integral cohomology of $\widetilde{X}$ has no torsion, but
$\widetilde{X}$ does not admit an integral cohomological
decomposition of the diagonal. A fortiori, $\widetilde{X}$ does not admit a Chow-theoretic
decomposition of the diagonal, that is, equivalently,  the group ${\rm
CH}_0(\widetilde{X})$ is not universally trivial.

\end{theo}
\begin{proof} The first statement is proved in \cite{endrass}, if
 we observe in the nodal case that $\widetilde{X}$ admits a unirational parametrization of degree
 $2$ (as all nodal quartic  double solids do, see paragraph \ref{sec01}). This indeed implies that the
 torsion in $H^*(\widetilde{X},\mathbb{Z})$ is of order $2$, while
 Endrass \cite{endrass}
 proves that there is no $2$-torsion in $H^*(\widetilde{X},\mathbb{Z})$ if $\widetilde{X}$ has less than $10$ nodes.

 We next claim the following:
 \begin{lemm}\label{tobeproved}  The general quartic double solid
$X$  with $k\leq 7$ nodes can be specialized to the Artin-Mumford
double solid $X_o$ constructed in \cite{artinmumford}. In
particular, its desingularization obtained by blowing-up its $k$
nodes can be specialized onto the partial desingularization of $X_o$
obtained by blowing-up the  corresponding $k$ nodes.
\end{lemm}

Postponing the proof of the lemma, we now conclude as follows: First
of all, as $\widetilde{X}_o$ has by Artin-Mumford
\cite{artinmumford} some nontrivial $2$-torsion in its integral
cohomology, it does not admit an integral cohomological
decomposition of the diagonal (see \cite{voisinJAG} or Theorem
\ref{theocitedeJAG}). We use now \cite{voisinintegralhodge} which
guarantees that the even degree integral cohomology of
$\widetilde{X}_o$ is algebraic, because $\widetilde{X}_o$ is
uniruled of dimension $3$. It then follows from Lemma
\ref{tobeproved}  and  Theorem \ref{theodeg} that the very general
$\widetilde{X}$ as in Corollary \ref{theocoro} does not admit an
integral cohomological decomposition of the diagonal.

\end{proof}
\begin{proof}[Proof of Lemma \ref{tobeproved}] The data of a $k$-nodal quartic double
solid is equivalent to the data of the corresponding quartic
ramification  surface which is also $k$-nodal. Let us consider a
general Artin-Mumford quartic surface $S$. It has $10$ nodes
$P_0,\ldots, P_9$, where  $P_0$ is the point defined in coordinates
$X_0,\ldots, X_3$ by $X_0=X_1=X_2=0$ and the $P_i,\,i\geq 1$ are
above $9$ points $O_i\in \mathbb{P}^2$ via the linear projection
$\mathbb{P}^3\dashrightarrow\mathbb{P}^2$ from $P_0$. The $9$ points
$O_i,\,i\geq 1$ form the reduced intersection of two plane cubics.
The general deformation theory of $K3$ surfaces (see
\cite{palaiseau}) tells us that the Artin-Mumford surface with $10$
nodes $P_0,\ldots,P_9$ can have its nodes smoothed independently,
keeping the other nodes. Let us  prove this statement in a more
algebraic and elementary way.
\begin{sublemma} \label{sublemma31jui} For any $k\leq 10$, there exist a smooth quasiprojective variety $B$, and
a family of quartic surfaces $\pi:\mathcal{S}\rightarrow
B,\,\mathcal{S}\subset B\times \mathbb{P}^3$ with the following
property: The general fiber of $\pi$ is $k$-nodal and there exists a
non-empty proper closed algebraic subset $B'\subset B$ of
codimension $10-k$ parameterizing $10$-nodal Artin-Mumford
surfaces\footnote{One can even show that the total space
$\mathcal{S}$ is smooth, but this is not useful here.}.
\end{sublemma}
\begin{proof} First of all, we claim that if $S\subset \mathbb{P}^3$ is
a $s+1$-nodal  quartic surface defined by a quartic polynomial $f$,
then the set $Z$ of nodes $P_i,\,i=0,\ldots , s$ of $S$ imposes
$s+1$
 independent conditions to quartic polynomials. Indeed,
one easily check that there exists an irreducible nodal curve in the
linear system $|\mathcal{I}_Z(4)|$ (since $|\mathcal{I}_Z(3)|$
contains the partial derivatives of $f$, this linear system has no
base point on $S\setminus Z$ and cuts $Z$ schematically, hence it is
 nef on the blow-up $\widetilde{S}$ of $S$ along $Z$; thus
$|\mathcal{I}_Z(4)|$ is nef and big  on $\widetilde{S}$). Then the
normalized curve $n:\widetilde{C}\rightarrow C$ contains the set
$\widetilde{Z}=n^{-1}(Z)$, and we have
$K_{\widetilde{C}}=n^*(\mathcal{O}_C(4))(-\widetilde{Z})$; thus, as
$\widetilde{C}$ is irreducible, we have $g(\widetilde{C})=g(C)-|Z|$,
which provides
$$h^0(\widetilde{C},n^*(\mathcal{O}_C(4))(-\widetilde{Z}))=g(\widetilde{C})=g(C)-|Z|=h^0(\mathcal{O}_C(4))-|Z|.$$
As  $H^0(\widetilde{C},n^*(\mathcal{O}_C(4))(-\widetilde{Z}))$
contains $H^0(C,\mathcal{O}_C(4)\otimes \mathcal{I}_Z)$, one
concludes that $$h^0(C,\mathcal{O}_C(4)\otimes \mathcal{I}_Z)\leq
h^0(\mathcal{O}_C(4))-|Z|$$ which proves the claim.

This implies classically that in the projective space $\mathbb{P}^N$
of all quartic homogeneous polynomials on $\mathbb{P}^3$, the
hypersurface $\mathcal{D}$ consisting of quartic polynomials with
one node has $s+1$ (in our case, $10$) analytic smooth branches
intersecting transversally at $f$. Concretely, the normalization
$\widetilde{\mathcal{D}}$ of $\mathcal{D}$ is defined as the
subvariety of $\mathbb{P}^3\times \mathbb{P}^N$ defined by
$$\widetilde{\mathcal{D}}=\{(x,g)\in \mathbb{P}^3\times \mathbb{P}^N,\,g\,\,{\rm is\,\, singular\,\, at}\,\, x\},$$
and the branches $\mathcal{D}_i$ of $\mathcal{D}$ passing through
$f$ are in one to one correspondence with the preimages of
$\widetilde{\mathcal{D}}\rightarrow \mathcal{D}$ over $f$, that is
the nodes $P_i$ of $S$.

 Coming back to our situation, the nodes
$P_1,\ldots,P_{k}$ being fixed, the intersection of the
corresponding analytic branches $\mathcal{D}_i,\,i=1,\ldots, k$ is
smooth and we can construct a smooth algebraic variety $B$
containing it as an analytic open set as follows: $B$ will be an
adequate  Zariski open neighborhood of $(P_1,\ldots,P_{k},f)$ in the
set
$$\{(x_1,\ldots,x_k,g)\in \mathbb{P}^3\times \mathbb{P}^N,\,g
\,\,{\rm is\,\, singular\,\, at}\,\, x_i,\,{\rm
for}\,\,i=1,\ldots,k\}.$$
 This variety $B$  maps naturally by the second projection to
$\mathbb{P}^N$, and the pull-back to $B$ of the universal family
$\mathcal{S}_{univ}\subset \mathbb{P}^N\times \mathbb{P}^3$
 provides us with a family
$$\mathcal{S}\rightarrow B,\,\mathcal{S}\subset \mathbb{P}^3\times B,$$
of quartic $K3$ surfaces, such that  the general fiber
$\mathcal{S}_b$ has $k$ nodes $P_{1,b},\ldots, P_{k,b}$, while the
fibers $\mathcal{S}_{b_0}$ for $b_0\in B'\subsetneqq B$ are
Artin-Mumford quartics with $10 $ nodes, $k$ of them being the
specialization of $P_{1,b},\ldots, P_{k,b}$. Here, with the notation
just introduced, the image of $B'$ in $\mathbb{P}^N$ is  the Zariski
closure of the intersection $\cap_{i=0}^9\mathcal{D}_i$ of all
branches.
 The argument above shows that the family is furthermore complete, that is, has $34-k$
parameters.
\end{proof}
Recall now from Remark \ref{rematartive} that there is an unique
irreducible component $M$ of the space $M'$ of nodal quartic
surfaces in $\mathbb{P}^3$ with $k\leq7$ nodes, dominating
$(\mathbb{P}^3)^{(k)}$ by the map $\phi:M\rightarrow
(\mathbb{P}^3)^{(k)}$ which to a $k$-nodal quartic associates its
set of nodes. Denote by $\psi:B\rightarrow M'$ the classifying map,
where $B$ is as in Sublemma \ref{sublemma31jui}. In order to prove
that $\psi(B)$ is Zariski open in $M$, which is the content of the
lemma, it suffices to show that the map $\phi\circ \psi:
B\rightarrow (\mathbb{P}^3)^{(k)}$ is dominating. Let us do it for
$k=7$, the other cases being similar and easier. We have the
following:
\begin{sublemma}\label{leuniquecomp} Let $S_t$ be a small
general deformation of $S$ with $7$ nodes  $P_{1,t},\ldots,
P_{7,t}$. Then the set of quartic surfaces which are singular at all
the points $P_{i,t}$ is of dimension $6$.
\end{sublemma}

\begin{proof} It suffices to prove the statement when $S_t$ is very general.
First of all we claim that  the Galois group of the cover
$\Sigma\rightarrow M$ parameterizing the nine singular points $O_i$
of the surface $S_m$, $m\in M$, where $M$ is the parameter space for
Artin-Mumford quartic surfaces,  acts on the set of $9$ points
$\{O_1,\ldots,O_9\}$ as the full symmetric group. This fact can be
proved by applying Harris' principle in \cite{harris}. Namely, one
just has to prove the following points:

 1) The Galois group of the function field of $M$ acts
bitransitively on the set $\{O_1,\ldots,O_9\}$. Equivalently,  the
variety $\Sigma\times_M\Sigma \setminus \Delta_\Sigma$ is
irreducible.

 2) The image contains transpositions, which appear as the local monodromy of the cover
 $\Sigma\rightarrow M$ at a point of simple
ramification.

The variety $M$ parameterizes the triples $(E_1,E_2,C)$, where $E_1$
and $E_2$ are plane cubics, and $C$ is a conic everywhere tangent to
$E_1$ and $E_2$ (see \cite{artinmumford}). The variety $\Sigma$
parameterizes the quadruples $(O,E_1,E_2,C)$ where $(E_1,E_2,C)\in
M$ and $O$ is one of the intersection points of $E_1$ and $E_2$. Let
us fix $C$ and the degree $3$ divisors  $D_1$ such that
$2D_1=E_1\cap C$, $D_2$ such that $2D_2=E_2\cap C$ of $C$. Then one
gets a subvariety $M_{C,D_1,D_2}\subset M$ and its inverse image
$\Sigma_{C,D_1,D_2}\subset \Sigma$. It clearly suffices to proves 1)
and 2) for the general cover $\Sigma_{C,D_1,D_2}\rightarrow
M_{C,D_1,D_2}$.  For the point 1), we
 project $\Sigma_{C,D_1,D_2}\times_{ M_{C,D_1,D_2}}\Sigma_{C,D_1,D_2}\setminus \Sigma_{C,D_1,D_2}$
 to $\mathbb{P}^2\times \mathbb{P}^2$ by the
 map $p$ which to
 $((O,E_1,E_2,C),(O',E_1,E_2,C))$ associates $(O,O')$. Observe now
 that
 $M_{C,D_1,D_2}$ is a Zariski open set in $\mathbb{A}^3\times \mathbb{A}^3$, the
 general element being of the form $(e_1+xc, e_2+yc)$ where $c$ is
 the equation of $C$, $e_1$ is given so that the restriction of $e_1$ to $C$ has divisor
 $2D_1$, $e_2$ is given so that  the restriction of $e_2$ to $C$ has divisor
 $2D_2$, and $x,\,y$ are two arbitrary homogeneous polynomials of
 degree $1$ on $\mathbb{P}^2$.
 The fiber of $p$ over a general couple of points $(O,O')$ then consists of
 the set of equations
 $(e_1+xc, e_2+yc)$ such that $e_1+xc$ and $e_2+yc$ vanish on $O$ and
 $O'$. This gives a system of four affine equations which has
 maximal rank except if $O$ or $O'$ belongs to $C$. But of
 course, the last situation does not occur generically on
 $\Sigma_{C,D_1,D_2}$, hence we conclude that for a dense Zariski
 open set $\Sigma_{C,D_1,D_2}^0$ of $\Sigma_{C,D_1,D_2}$,
$\Sigma_{C,D_1,D_2}^0\times_M\Sigma_{C,D_1,D_2}^0 \setminus
\Delta_{\Sigma_{C,D_1,D_2}^0}$ is irreducible.

For the point 2), as the equations above clearly show that
$\Sigma_{C,D_1,D_2}$ is smooth at a point $(O,x,y)$ where $O$ does
not belong to $C$, it suffices to show that there exists such a
point $(O,x,y)\in \Sigma_{C,D_1,D_2}$ with $E_1,\,E_2$ meeting
tangentially  at $O$ and transversally at the other remaining $7$
points. For this, we fix the point $O$ not on $C$, and fix $E_1$
(with equation $e_1+xc$) passing through $O$. We then  look at the
set of equations $e_2+yc$ vanishing at $O$ and tangent to $E_1$ at
$O$. The restriction of these equations provides a linear system of
dimension $2$ on $E_1$, and one easily checks that for general
choice of $C,\,D_1,\,D_2,\, E_1,\, O$, its base-locus is reduced to
the point $O$ with multiplicity $2$. By Bertini, the general
intersection $E_1\cap E_2$ for $E_2$ as above has only the point $O$
for double point.
 This proves the claim.

 One easily deduces from this that for any choice of
$P_1,\ldots, P_7$, the classes $e_i$ of the corresponding
exceptional curves
 of the minimal desingularization $\widetilde{S}$ of $S$ and
the class $h=c_1(\mathcal{O}_S(1))$  generate a primitive sublattice
of $H^2(\widetilde{S},\mathbb{Z})$ (equivalently, there are no
relations with coefficients in $\mathbb{Z}/2\mathbb{Z}$ between
these classes). Hence for the very general deformation $S_t$ as
above, its Picard group is freely generated by the classes
$e_i,\,i=1,\ldots,7$, and $h$. Let $\tau:X\rightarrow \mathbb{P}^3$
be the blow-up of $\mathbb{P}^3$ at the points $P_{1,t},\ldots,
P_{7,t}$, with exceptional divisors $D_{1,t},\ldots,D_{7,t}$
intersecting the proper transform $\widetilde{S}_t$ along $E_{i,t}$.
The surface $\widetilde{S}_t$ belongs to the linear system $|L|$,
$L:=\tau^*(\mathcal{O}_{\mathbb{P}^3}(4))(-2\sum_iD_{i,t})$ and we
want to prove that ${\rm dim}\,|L|= 6$. As $H^1(X,\mathcal{O}_X)=0$,
this is equivalent to saying that $h^0(\widetilde{S}_t, L_{\mid
\widetilde{S}_t})= 6$, and also to $h^1(\widetilde{S}_t, L_{\mid
\widetilde{S}_t})=0$. Let $L_t:=L_{\mid \widetilde{S}_t}$. As
$L_{t}$ is big, this last vanishing is satisfied if $L_t$ is
numerically effective, hence if the linear system $|L_{t}|$ has no
base curve on which $L_t$ has negative degree, which is equivalent
(see \cite[Chapter 2, 1.6]{refk3}) to saying that there is no smooth
rational curve $C\subset \widetilde{S}_t$ such that
\begin{eqnarray} \label{numcond}L_{t}.C<0,\,\, {\rm and}\,\,\,C^2=-2
\end{eqnarray} with $ L_{t}(-C)$ effective. Let $C$ be such a curve,
and write its class $c\in H^2(\widetilde{S}_t,\mathbb{Z})$ as
$$c=\lambda h+\sum_in_ie_i,$$
 with $\lambda,\,n_i$ integers. Furthermore $n_i\leq 0$ as otherwise $C$ has to be one of the
 $E_{i,t}$, and does  not satisfy the condition $L_{t}.C<0$.
 The two numerical conditions (\ref{numcond}) write
\begin{eqnarray} \label{numcond0}
 4\lambda^2-2\sum_{i=1}^{7}n_i^2=-2,\,\,
16 \lambda+4\sum_{i=1}^{7}n_i<0.
\end{eqnarray}
 Of course one has $4>\lambda >0$ because $C$ is effective and $L_{t}(-C)$ is
 effective.
 In fact, the case $\lambda=3$ is impossible
 because the linear system $|L_{t}(-C)|$ has dimension
 $\geq 5$. So only $\lambda=1,\,2$ are possible.

 For $\lambda=1$ we get from (\ref{numcond0})
\begin{eqnarray}\label{eqncurves1}
 2-\sum_{i=1}^{7}n_i^2=-1,\,\,
4+\sum_{i=1}^{7}n_i<0.
 \end{eqnarray}
 and for
 $\lambda=2$ we get
  \begin{eqnarray}\label{eqncurves2}
 8-\sum_{i=1}^{7}n_i^2=-1,\,\,
8 +\sum_{i=1}^{7}n_i<0.
 \end{eqnarray}
 It is easy to check that neither (\ref{eqncurves1}) nor (\ref{eqncurves2}) has an integral solution with all $n_i$'s $\leq0$.
and this concludes the proof of the  sublemma.

\end{proof}

  Sublemma \ref{leuniquecomp} tells us that the fibers of
 $\phi\circ \psi$ are at most $6$-dimensional. Hence $${\rm
 dim}\,({\rm Im}\,\phi\circ \psi)=21={\rm
 dim}\,(\mathbb{P}^3)^{(7)},$$ which concludes the proof of Lemma \ref{tobeproved}.

\end{proof}
We  conclude this section with the following result which concerns
property \ref{item4} of Theorem \ref{theocitedeJAG} and more
generally property (*):
\begin{theo} \label{theotrue} (cf. Theorem \ref{theointrocoh30jui}, (ii)) Let $\widetilde{X}$ be the natural  desingularization of
a general quartic double solid $X$ with $ 7$ nodes. Then
$\widetilde{X}$ admits no universal codimension $2$ cycle $Z\in
CH^2(J^3(\widetilde{X})\times \widetilde{X})$.

More precisely, there is no smooth connected projective variety $B$
equipped with a codimension $2$ cycle $Z\in CH^2(B \times
\widetilde{X})$ which is cohomologous to $0$ on fibers $b\times X$,
$b\in B$, and such that the morphism $\Phi_Z:B\rightarrow
J^3(\widetilde{X})$ induced by the Abel-Jacobi map of
$\widetilde{X}$ is surjective with rationally connected general
fibers.

\end{theo}
\begin{proof} We use the following strengthening of Theorem \ref{theocitedeJAG}    (this is \cite[Theorem  4.9]{voisinJAG}   combined with
the result of \cite{voisinintegralhodge} guaranteeing the
algebraicity of $H^{4}(Y,\mathbb{Z})$ for $Y$ a rationally connected
threefold):
 \begin{theo}
Let $Y$ be a rationally connected  $3$-fold  satisfying the
following
 properties:

 (a) $H^*(Y,\mathbb{Z} ) $ has no torsion;

 (b) There is  a codimension $2$-cycle $Z\in {\rm CH}^2(B\times Y)$ inducing a surjective
 map
$\Phi_Z:B\rightarrow J^3(Y)$ with rationally connected general
fibers;

(c)   $J^3(Y)$ has a $1$-cycle in the minimal cohomology class
$\theta^{g-1}/(g-1)!$, $g={\rm dim}\,J^3(Y)$.

 Then $Y$  admits an integral  cohomological
decomposition of the diagonal.

\end{theo}

 Take now for $Y$ the desingularization $\widetilde{X}$ of
the general double solid   $X$ with $7$ nodes. Then  property (a)
holds as  already mentioned. The property ${\rm
dim}\,J^3(\widetilde{X})=3$ is satisfied for any double solid
$X\rightarrow \mathbb{P}^3$ ramified along a nodal quartic with $7$
nodes imposing independent conditions to quadrics in $\mathbb{P}^3$,
see \cite[Corollary 2.32]{clemens}.   Property (c) is thus satisfied
in our case because the set $z$ of nodal points is general, so that
${\rm dim}\,J^3(\widetilde{X})=3$, and any ppav of dimension $3$ is
a Jacobian. As Theorem \ref{theocoro} says that $\widetilde{X}$ does
not admit an integral  cohomological decomposition of the diagonal,
we conclude that (b) must fail.
\end{proof}
\section{Application to the  universal  degree $3$ unramified cohomology with torsion
coefficients \label{sec2}}

We  prove in this section   Corollary \ref{cornew30juill}, that we
will get as a direct consequence of the following result (cf.
Theorem \ref{theounramintro}):
\begin{theo}\label{theounram} Let $X$ be a smooth complex projective variety of dimension
$n$ with ${\rm CH}_0(X)=\mathbb{Z}$.
Assume
\begin{enumerate}
\item \label{hytor} $H^*(X,\mathbb{Z})$ has no torsion and the K\"unneth components
$\delta_0,\ldots,\delta_4$ of the diagonal
 are algebraic.
\item \label{hy3} The group $H^3_{nr}(X,\mathbb{Q}/\mathbb{Z})$ is trivial
(or equivalently by \cite{CTvoisin}, the integral Hodge classes of degree $4$ on $X$ are algebraic).
\end{enumerate}
Then the degree $3$ unramified cohomology of $X$ is universally
trivial if and only if there is a universal codimension $2$ cycle
$Z\in {\rm CH}^2(J^3(X)\times X)$.
\end{theo}
Here and below, the K\"unneth components $\delta_i$ act as identity
on $H^i(X,\mathbb{Z})$, and as $0$ on $H^j(X,\mathbb{Z})$ for
$i\not=j$. They are well-defined because $H^*(X,\mathbb{Z})$ is
torsion-free.
\begin{example}\label{exacubic}
{\rm This theorem applies for example to cubic $4$-folds. Indeed,  by \cite{voisinhodge}, they satisfy
 the integral Hodge conjecture in degree $4$ and the other assumptions
 are easy to check. Obviously, they have trivial ${\rm CH}_0$ group.
 The fact that there is no torsion in the cohomology
 of a smooth hypersurface in projective space is a consequence of the Lefschetz theorem on hyperplane sections.
 Finally,  the K\"unneth components of their
 diagonal are algebraic, because their cohomology groups of degree $6$ and $2$ are
 cyclic,
 generated by the class $\gamma$ of a line and the class $h$ of a hyperplane
 section respectively. Thus the components $\delta_2$ and
 $\delta_6$ which are the projectors on $H^2$ and $H^6$ respectively
 are given by $\delta_2=\gamma\otimes h$ and
 $\delta_6=h\otimes \gamma$. It follows that the remaining K\"unneth
 component
 $$\delta_4=[\Delta_X]-\delta_0-\delta_8-\delta_2-\delta_6$$
 is algebraic.

  As their intermediate Jacobian is trivial, one concludes by Theorem \ref{theounram} that their  third unramified
  cohomology with torsion coefficients is universally trivial. This generalizes the main result of
  \cite{auelCT} with a completely different proof. Auel,
  Colliot-Th\'el\`ene and Parimala
   prove that
   the  unramified cohomology of degree $3$ with torsion coefficients of a very general
  cubic fourfold $X$ containing a plane is universally trivial. Their method uses the $K$-theory of quadric bundles.
  }
  \end{example}
  We get similarly
  \begin{coro} \label{corofififigeneral} Let $X$ be a rationally connected threefold with no torsion in
  $H^3(X,\mathbb{Z})$.  Then the
 third unramified cohomology of $X$ with coefficients in
$\mathbb{Q}/\mathbb{Z}$ is not universally trivial if and only if
$X$ does not admit a universal codimension $2$ cycle.

  \end{coro}
\begin{proof} We just have to check the assumptions of Theorem
\ref{theounram}. As $X$ is rationally connected, we clearly have
${\rm CH}_0(X)=\mathbb{Z}$ and furthermore, the fact that
$H^3(X,\mathbb{Z})$ has no torsion implies that the whole integral
cohomology $H^*(X,\mathbb{Z})$ has no torsion. The K\"unneth
components $\delta_2$ and $\delta_4$ of the diagonal of $X$ (which
act by projection on $H^2(X,\mathbb{Z})$ and $H^4(X,\mathbb{Z})$
respectively) are algebraic because $H^2(X,\mathbb{Z})$ and
$H^4(X,\mathbb{Z})$ are generated by algebraic classes. (Note that
for the cohomology group $H^4(X,\mathbb{Z})$, the fact that it is
generated by classes of curves is not obvious and proved in
\cite{voisinintegralhodge}.)  Thus the last component
$$\delta_3=[\Delta_X]-\delta_0-\delta_2-\delta_4-\delta_6$$
is also algebraic. Finally, assumption \ref{hy3} in Theorem
\ref{theounram}  reduces again to the fact already mentioned that
$H^4(X,\mathbb{Z})$ is algebraic.
\end{proof}
\begin{coro}\label{corofififi} (cf. Corollary \ref{cornew30juill}) Let $X$ be  the  natural  desingularization   of
a very general quartic double solid  with $ 7$ nodes. Then the third
unramified cohomology group of $X$ with torsion coefficients is not
universally trivial.
\end{coro}
\begin{proof} Indeed,   there is
no torsion in  $H^3(X,\mathbb{Z})$; this  has been already mentioned
before  and is proved by Endrass \cite{endrass}. As we know by
Theorem \ref{theotrue} that the desingularization of the very
general double solid $X$ with $7$ nodes does not admit a universal
codimension $2$ cycle, the corollary is thus  a consequence of
Corollary  \ref{corofififigeneral}.

\end{proof}

\begin{proof}[Proof of Theorem \ref{theounram}]   Let us first show
that if $X$ has trivial ${\rm CH}_0$ group, satisfies the
assumptions \ref{hytor} and   \ref{hy3} of the theorem and has no
universal codimension $2$ cycle, then it has a nontrivial universal
third unramified cohomology group with torsion coefficients. We
recall from the introduction that the meaning of this statement is
that there exist a smooth quasi-projective variety $U$, and an
unramified cohomology class $\alpha\in H^3_{nr}(U\times
X,\mathbb{Q}/\mathbb{Z})$ with the property that for any  Zariski
dense open subset $U'\subset U$, $\alpha_{\mid U'\times X}$ is not
the pull-back of a cohomology class $\beta\in
H^3(U',\mathbb{Q}/\mathbb{Z})$.

The fact that  $H^*(X,\mathbb{Z})$ has no torsion implies that there
is a K\"unneth decomposition of cohomology with integral
coefficients of $U\times X$  for any variety $U$. On the other hand,
the fact that the K\"unneth components $\delta_i,\, 0\leq i\leq 4$,
of the diagonal of $X$ (which are defined in integral coefficients
cohomology) are algebraic implies that for any $U$ and any algebraic
cycle $z$ of codimension $\leq 2$ on $U\times X$, the K\"unneth
components of $[z]$ are algebraic, since they are obtained by
applying the correspondences $\delta_i,\,0\leq i\leq 4$, seen as
relative self-correspondences of $U\times X\times X$ over $U$, to
$[z]$.

According to \cite{CTvoisin}, where the result is stated for smooth
projective varieties but works in the smooth quasi-projective case
as well, there is an exact sequence, for any smooth quasi-projective
$Y$: \begin{eqnarray} \label{eqlongshort1eraout}0\rightarrow
H^3_{nr}(Y,\mathbb{Z})\otimes\mathbb{Q}/\mathbb{Z}\rightarrow
H^3_{nr}(Y,\mathbb{Q}/\mathbb{Z})\rightarrow {\rm
Tors}(H^4(Y,\mathbb{Z})/H^4(Y,\mathbb{Z})_{alg})\rightarrow 0,
\end{eqnarray}
where $H^4(Y,\mathbb{Z})_{alg}\subset H^4(Y,\mathbb{Z})$ is the
subgroup of cycle classes $[Z],\,Z\in {\rm CH}^2(Y)$. Using this
exact sequence, in order to prove that the third unramified
cohomology group of $X$ is not universally trivial, it suffices to
exhibit a smooth projective variety $B$, and a degree $4$ cohomology
class $\alpha$ on $B\times X$, such that $N\alpha$ is algebraic for
some $N\not=0$, but $\alpha_{\mid U\times X}$ cannot be written as a
sum $a+pr_1^*b$, with $a$ algebraic on $U\times X$ and $b\in
H^4(U,\mathbb{Z})_{tors}$, for any dense Zariski open set $U\subset
B$. Such a class indeed provides a torsion class $\overline{\alpha}$
in $H^4(B\times X,\mathbb{Z})/H^4(B\times X,\mathbb{Z})_{alg}$; by
the exactness on the right in (\ref{eqlongshort1eraout}), there
exists a lift $\tilde{\overline{\alpha}}\in H^3_{nr}(B\times
X,\mathbb{Q}/\mathbb{Z})$; then $\tilde{\overline{\alpha}}_{\mid
U\times X}$ is not in $pr_1^*H^3(U,\mathbb{Q}/\mathbb{Z})$ for any
$U\subset B$ dense Zariski open. Indeed, we have a commutative
diagram

\begin{eqnarray}\label{numerodiag}
 \label{diagram} \xymatrix{
&H^3(U,\mathbb{Q}/\mathbb{Z})\ar[r]\ar[d]&H^4(U,\mathbb{Z})_{tors}\ar[d]&&\\
&H^3(U\times X,\mathbb{Q}/\mathbb{Z})\ar[r]\ar[d]&H^4(U\times
X,\mathbb{Z})_{tors}\ar[d]&&\\
&H^3_{nr}(U\times X,\mathbb{Q}/\mathbb{Z})\ar[r]&H^4(U\times
X,\mathbb{Z})/H^4(U\times X,\mathbb{Z})_{alg}&&}
\end{eqnarray}
where the first two vertical maps are pull-back maps $pr_1^*$,  and
the first two horizontal ones are induced by the exact sequence
$$0\rightarrow \mathbb{Z} \rightarrow\mathbb{Q}\rightarrow
\mathbb{Q}/\mathbb{Z}\rightarrow0.$$ The third horizontal map is the
last map of  (\ref{eqlongshort1eraout}) for $Y=U\times X$. So if
$\tilde{\overline{\alpha}}_{\mid U\times X}$ belonged to
$pr_1^*H^3(U,\mathbb{Q}/\mathbb{Z})$, its image in $H^4(U\times
X,\mathbb{Z})/H^4(U\times X,\mathbb{Z})_{alg}$, which is also the
restriction of $\overline{\alpha}$ to $U\times X)$, would come from
a torsion class in $H^4(U,\mathbb{Z})$, which we have excluded.

We now  construct $B$ and $\alpha$: We take  $B:=J^3(X)$. We first prove:
\begin{lemm} \label{letardif30jui} There is an integer $N$, and a codimension $2$ cycle
$Z_N\in {\rm CH}^2(J^3(X)\times X)$ such that ${Z_N}_{\mid t\times
X}$ is cohomologous to $0$ for any $t\in J^3(X)$ and
$$\Phi_{Z_N}:J^3(X)\rightarrow J^3(X),\,t\mapsto AJ_X(Z_{N,t})$$
is equal to $N\,Id_{J^3(X)}$.
\end{lemm}
\begin{proof} Indeed, as we know that ${\rm
CH}_0(X)=\mathbb{Z}$, the Abel-Jacobi map $$\Phi_X:{\rm
CH}^2(X)_{hom}\rightarrow J^3(X)$$ is surjective (see
\cite{blochsrinivas}). Hence there exist a variety $W$ and a
codimension $2$ cycle $Z\in {\rm CH}^2(W\times X)$ such
that
$$\Phi_Z:W\rightarrow J^3(X),\,\Phi_Z(t)=AJ_X(Z_t),$$
is a surjective morphism of algebraic varieties. Replacing $W$ by a
linear section if necessary, we may assume that $\Phi_Z:W\rightarrow
J^3(X)$ is generically finite of degree $N$. We now take for $Z_N$
the cycle $(\Phi_{Z},Id_X)*(Z)$. It is immediate to check that
$\Phi_{Z_N}:J^3(X)\rightarrow J^3(X)$ is equal to $N\,Id_{J^3(X)}$.
 Thus the lemma is
proved.
\end{proof}
 The condition that $\Phi_{Z_N}:J^3(X)\rightarrow J^3(X)$
is equal to $N\,Id_{J^3(X)}$ is equivalent to the  fact that the
K\"unneth component of type $(1,3)$ of $[Z_N]$ induces $N$ times the
canonical isomorphism between $H_1(J^3(X),\mathbb{Z})$ and
$H^3(X,\mathbb{Z})$. Furthermore, as explained above, by taking the
K\"unneth component of type $(1,3)$, we may assume $[Z_N]$ is of
K\"unneth type $(1,3)$.
 Recall now from Remark \ref{remahodgeclass}  that there
is also an integral (Hodge) class $\alpha$ on $J^3(X)\times X$ which
is of K\"unneth type $(1,3)$ and induces the canonical isomorphism
between $H_1(J^3(X),\mathbb{Z})$ and $H^3(X,\mathbb{Z})$. We thus
have
$$[Z_N]=N\alpha\,\,\,{\rm in}\,\,H^{4}(J^3(X)\times X,\mathbb{Z}).$$
Hence we constructed the class $\alpha$ and to finish we just have
to prove the following:
\begin{lemm} For any dense Zariski open set $U\subset J^3(X)$, the image of $\alpha$ in
$$H^4(U\times X,\mathbb{Z})/H^4(U\times X,\mathbb{Z})_{alg}$$ is
nonzero modulo $pr_1^* H^4(U,\mathbb{Z})_{tors}$.
\end{lemm}
\begin{proof} Otherwise
$\alpha_{\mid U\times X}=a+pr_1^*b$, where $a$ is algebraic and $b$
is torsion, for some dense Zariski open set $U\subset J^3(X)$. But
$\alpha_{\mid U\times X}$ is of K\"unneth  type $(1,3)$ while
$pr_1^*b$ is of K\"unneth  type $(4,0)$. As  the K\"unneth
decomposition, which works as well on $U$, preserves algebraic
classes, we conclude  by projection on the K\"unneth component of
type $(1,3)$ that $\alpha_{\mid U\times X}=a$ is algebraic on
$U\times X$. (Note that, alternatively, we can restrict to a smaller
Zariski open set, where the torsion class $b$ vanishes.)
 This means
that there is a decomposition
$$\alpha=\alpha_1+\alpha_2$$
in $H^{4}(J^3(X)\times X,\mathbb{Z})$, where $\alpha_1$ is the class
of a cycle $Z$ on $J^3(X)\times X$ and $\alpha_2$ is a cohomology
class supported on $D\times X$, where $D=J^3(X)\setminus U$. But
then the K\"unneth component of type $(1,3)$ of $\alpha_2$ must be
$0$. So the K\"unneth component of type $(1,3)$ of $Z$ is equal to
$\alpha$ and $Z$ is a universal codimension $2$ cycle on
$J^3(X)\times X$, contradicting our assumption.

\end{proof}

In the other direction, let us prove that if $X$ satisfies the
assumptions of Theorem \ref{theounram} and has a universal
codimension $2$ cycle, then its third unramified cohomology group
with torsion coefficients is universally trivial.

Let thus $B$ be a smooth quasi-projective complex variety and let
$\overline{B}$ be a smooth projective completion of $B$. Let
$\alpha\in H^3_{nr}(B\times X,\mathbb{Q}/\mathbb{Z})$. We want to
show that up to shrinking $B$ to a Zariski open set $B'$, $\alpha
=pr_1^*\beta$ for some class $\beta\in
H^3(B',\mathbb{Q}/\mathbb{Z})$. Note that $H^3_{nr}(B\times
X,\mathbb{Q}/\mathbb{Z})$ contains $H^3_{nr}(B\times
X,\mathbb{Z})\otimes \mathbb{Q}/\mathbb{Z}$.
\begin{lemm}\label{lemmapresquefin} If $X$ has trivial ${\rm CH}_0$ group,
the map $pr_1^*:H^3_{nr}(B,\mathbb{Z})\rightarrow  H^3_{nr}(B\times
X,\mathbb{Z})$ is an isomorphism.
\end{lemm}
\begin{proof}  Indeed, this map has a left inverse, namely the
restriction map $r_{B\times x}$ to $B\times x$ for any point $x\in
X(\mathbb{C})$. Furthermore it follows from the Merkurjev-Suslin
theorem (or the Bloch-Kato conjecture) that the groups $
H^3_{nr}(B,\mathbb{Z})$ and $H^3_{nr}(B\times X,\mathbb{Z})$ have no
torsion. Thus it suffices to show that the map
$pr_1^*:H^3_{nr}(B,\mathbb{Q})\rightarrow
 H^3_{nr}(B\times
X,\mathbb{Q})$ is an isomorphism. This is however implied by the
Bloch-Srinivas decomposition of the diagonal (\ref{decompintro}) for
some integer $N\not=0$, which is satisfied by $X$ since ${\rm
CH}_0(X)=\mathbb{Z}$. This gives as well a decomposition over $B$:
\begin{eqnarray}\label{decompintronewB}
N(B\times \Delta_X)=Z_1+Z_2\,\,{\rm in}\,\, {\rm CH}^n(B\times
X\times X),\,n={\rm dim}\,X,
\end{eqnarray} where $Z_2=N(B\times X\times x)$ for some point $x\in X(\mathbb{C})$ and $Z_1$ is supported on
$B\times D\times X$, for some proper closed algebraic subset
$D\varsubsetneqq X$. As recalled  in \cite[Appendix]{CTvoisin}, the
various cycles $\gamma$ appearing in this equality act on unramified
cohomology of $B\times X$ via
$$\eta\mapsto \gamma^*\eta:= pr_{B,1*}(pr_{B,2}^*\eta\cdot [\gamma]_{BO}),$$
where the class $[\gamma]_{BO}\in H^n((B\times X\times
X)_{Zar},\mathcal{H}^n(\mathbb{Z}))$ is the Bloch-Ogus cycle class
of $\gamma$ (see \cite{blochogus}), and $pr_{B,1},\, pr_{B,2}$ are
the two projections from $B\times X\times X$ to $B\times X$. As
$B\times \Delta_X$ acts as identity on $H^3_{nr}(B\times
X,\mathbb{Q})$ and $Z_{1}^* $ acts as $0$ on $H^3_{nr}(B\times
X,\mathbb{Q})$, one gets that
$$N\,Id_{H^3_{nr}(B\times X,\mathbb{Q})}=Z_2^*=N(pr_1^*\circ r_{B\times
x}),$$ which proves the result.

\end{proof}
From Lemma \ref{lemmapresquefin}, we see  that it suffices to show
that, up to passing to a dense Zariski open subset of $B$ if
necessary, the image of $\alpha$ in the quotient
$$\frac{H^3_{nr}(B\times
X,\mathbb{Q}/\mathbb{Z})}{H^3_{nr}(B\times X,\mathbb{Z})\otimes
\mathbb{Q}/\mathbb{Z}}$$ is $0$. By the exact sequence
(\ref{eqlongshort1eraout}), this quotient is isomorphic to  the
torsion of the group $H^4(B\times X,\mathbb{Z})/H^4(B\times
X,\mathbb{Z})_{alg}$. So the result follows from the following:
\begin{lemm} Let $X$ satisfy the assumptions of Theorem
\ref{theounram}. Then, if $X$ has a universal codimension $2$ cycle,
for any $B$ and any degree $4$ class $\alpha\in H^4(B\times
X,\mathbb{Z})$, such that $N\alpha$ is algebraic for some integer
$N\not=0$, there is  a dense  Zariski open subset $B'\subset B$,
such that  the class $\alpha$ is algebraic in $B'\times X$.
\end{lemm}
\begin{proof} First of all, as the integral cohomology of $X$ is
torsion free, there is a K\"unneth decomposition
$$\alpha=\sum_{0\leq i\leq4}\alpha_{i,4-i},$$
where $\alpha_{i,4-i}\in H^i(B,\mathbb{Z})\otimes
H^{4-i}(X,\mathbb{Z})$. Furthermore, this decomposition is obtained
by applying to $\alpha$ the  K\"unneth projectors $\delta_i,\,0\leq
i\leq 4$ of $X$. As we know that the $\delta_i$'s are algebraic for
$i\leq 4$, each class $\alpha_{i,4-i}$ satisfies the property that
$N\alpha_{i,4-i}$ is algebraic.

Consider first the term $\alpha_{4,0}\in pr_1^*H^4(B,\mathbb{Z})$
and write it $\alpha_{4,0}=pr_1^*\beta,\,\beta\in
H^4(B,\mathbb{Z})$. Then $N\beta$ is algebraic on $B$, so there is a
dense Zariski open set $B'\subset B$ such that $N\alpha=0$ in
$H^4(B',\mathbb{Z})$. The Bloch-Kato conjecture  then implies that
there is a dense Zariski open set $B''\subset B$ such that
$\alpha=0$ in $H^4(B'',\mathbb{Z})$ since it implies that the sheaf
$\mathcal{H}^4(\mathbb{Z})$ on $B_{Zar}$ has no
$\mathbb{Z}$-torsion.

Next consider the class $\alpha_{2,2}\in p^*H^2(B,\mathbb{Z})\otimes
H^2(X,\mathbb{Z})$ and write it
$$\alpha_{2,2}=\sum_ipr_1^*\beta_i\smile pr_X^* \alpha_i,\,\beta_i\in
H^2(B,\mathbb{Z}),$$ where $pr_X:B\times X\rightarrow X$ is the
second projection and the $\alpha_i$'s form a basis of the free
abelian group $H^2(X,\mathbb{Z})$, which is generated by divisors
classes since ${\rm CH}_0(X)=\mathbb{Z}$. The group
$H^{2n-2}(X,\mathbb{Z})$ admits a Poincar\'e dual basis
$\alpha_i^*$, with $<\alpha_i,\alpha_j^*>=\delta_{ij}$.
  The $\mathbb{Q}$-vector $H^{2n-2}(X,\mathbb{Q})$ is
generated by classes of curves, since $H^2(X,\mathbb{Q})$ is
generated by divisor classes and for any ample line bundle $L$ on
$X$, the topological Chern class $l=c_{1,top}(L)\in
H^2(X,\mathbb{Q})$ provides a Lefschetz isomorphism $$
l^{n-2}\smile:H^2(X,\mathbb{Q})\cong H^{2n-2}(X,\mathbb{Q}).$$
 Thus there exists a nonzero integer $N'$ such that
 $N'\alpha_i^*=[z_i]$ in $H^{2n-2}(X,\mathbb{Z})$
  for some $1$-cycles $z_i\in
CH^{n-1}(X)$. It thus follows that
$$\beta_i={pr_1}_*(\alpha_{2,2}\smile
pr_X^*\alpha_i^*),\,N'\beta_i={pr_1}_*(\alpha_{2,2}\smile
pr_X^*[z_i]).$$ As $N\alpha_{2,2}$ is algebraic, so are the classes
$NN'\beta_i$. This implies that the classes $\beta_i$ themselves are
algebraic, by the same argument as before since they restrict to
torsion classes on some Zariski open set $B'\subset B$, and this
implies that they have to vanish identically on a dense Zariski open
set $B''\subset B$.

The term $\alpha_{0,4}\in pr_X^*H^4(X,\mathbb{Z})\cong
H^4(X,\mathbb{Z}) $  satisfies the condition that $N\alpha_{0,4}$ is
algebraic. The assumption \ref{hy3} made on $X$ implies that
$\alpha_{0,4}$ itself is algebraic.

As $H^1(X,\mathbb{Z})=0$, we are thus left with $\alpha_{1,3}\in
H^1(B,\mathbb{Z})\otimes H^3(X,\mathbb{Z})$. We have now:
\begin{lemm} \label{lefinfinfinfin}
(i) The class $\alpha_{1,3}$ is the restriction to $B\times X$ of a
class $\overline{\alpha}_{1,3}\in
H^1(\overline{B},\mathbb{Z})\otimes H^3(X,\mathbb{Z})$ which has the
property that $N\overline{\alpha}_{1,3}$ is algebraic.

(ii)  There is a morphism $\phi: \overline{B}\rightarrow J^3(X)$
such that $(\phi,Id_X)^*\alpha=\overline{\alpha}_{1,3}$, where
$\alpha$ is the integral Hodge class on $J^3(X)\times X$ introduced
in Remark \ref{remahodgeclass}.
\end{lemm}
\begin{proof} (i)  The class $N\alpha_{1,3}$ extends to an integral
cohomology class on $\overline{B}\times X$ because it is algebraic.
As the K\"unneth components of the diagonal of $X$ are algebraic,
one may even assume that the class $N\alpha_{1,3}$ extends to an
integral cohomology class $\beta$ on $\overline{B}\times X$  which
is algebraic and of K\"unneth type $(1,3)$. As $H^*(X,\mathbb{Z})$
has no torsion, this class $\beta$ can be seen as a morphism
$$\beta_*:H_3(X,\mathbb{Z})\rightarrow H^1(\overline{B},\mathbb{Z})$$
which has the property that, denoting by
$r_B:H^1(\overline{B},\mathbb{Z})\rightarrow H^1({B},\mathbb{Z})$
the restriction to $B$, the composite morphism
$$r_B\circ
\beta_*=N{\alpha_{1,3}}_*:H_3(X,\mathbb{Z})\rightarrow
H^1({B},\mathbb{Z})$$ is divisible by $N$. On the other hand, it is
quite easy to prove  that the restriction map $r_B$ is injective and
that its cokernel is torsion free. Thus the morphism $\beta_*$ is
also (uniquely) divisible by $N$, and so is $\beta$, which proves
(i).

(ii)  The class $N\overline{\alpha}_{1,3}$ being algebraic, it is an
integral Hodge class on $\overline{B}\times X$, so
$\overline{\alpha}_{1,3}$ is also an integral Hodge class. But the
morphisms from $\overline{B}$ to $J^3(X)$ identify  (modulo the
translations of $J^3(X)$) to the morphisms of complex tori between
${\rm Alb}\,\overline{B}$ and $J^3(X)$ which themselves identify
 to the morphisms of Hodge structures
$$H_1(\overline{B},\mathbb{Z})/{\rm torsion}\rightarrow
H^3(X,\mathbb{Z})$$ because $H^3(X,\mathbb{Z})$ has no torsion, and
finally these morphisms of Hodge structures identify to the integral
Hodge classes of K\"unneth type $(1,3)$ on $\overline{B}\times X$.

Hence the class $\overline{\alpha}_{1,3}$ provides us with a
morphism
$$\phi
: \overline{B}\rightarrow J^3(X)$$
and it is a formal fact to prove following  the chain of
identifications above that
$$(\phi,Id_X)^*(\alpha)=\overline{\alpha}_{1,3}.$$

\end{proof}
The proof is now finished because we assumed that $X$ admits a
universal codimension $2$ cycle. This is equivalent to saying that
the class $\alpha$ of Remark \ref{remahodgeclass} is algebraic
because $H^3(X,\mathbb{Z})$ has no torsion. Lemma
\ref{lefinfinfinfin}  thus implies that $\alpha_{1,3}$ is algebraic.
\end{proof}
The proof of Theorem \ref{theounram} is finished.
\end{proof}

Centre de math\'ematiques Laurent Schwartz

91128 Palaiseau C\'edex, France

\smallskip
 voisin@math.jussieu.fr
    \end{document}